\documentclass[11pt]{article}
\usepackage[left=3cm,right=3cm,top=3cm,bottom=3cm]{geometry} % page settings
\usepackage{amsmath,amssymb}
\usepackage{amsthm}
\usepackage{graphicx}
\usepackage{float}
\usepackage{caption}
\usepackage{subcaption}
\usepackage{setspace}
\usepackage{tikz}
\usepackage{verbatim}
\newcommand*{\rom}[1]{\expandafter\@slowromancap\romannumeral #1@}
\newcommand{\sA}{\mathcal A}

\newcommand{\sC}{\mathcal C}

\newcommand{\sF}{\mathcal F}
\newcommand{\sG}{\mathcal G}
\newcommand{\sI}{\mathcal I}
\newcommand{\sL}{\mathcal L}

\newcommand{\sS}{\mathcal S}
\newcommand{\sT}{\mathcal T}
\newcommand{\Tau}{\mathcal T}

\newcommand{\sX}{\mathcal X}
\newcommand{\R}{\mathbb R}
\newcommand{\E}{\mathbb E}
\newcommand{\F}{\mathbb F}
\newcommand{\G}{\mathbb G}
\newcommand{\Prob}{\mathbb P}

\newcommand{\ul}{\underline{\lambda}}
\newcommand{\ol}{\overline{\lambda}}
\newcommand{\hl}{\hat{\lambda}}

\newtheorem{thm}{Theorem}

\newtheorem{prop}{Proposition}

\newtheorem{lem}{Lemma}
\newtheorem{cor}{Corollary}
\newtheorem{rem}{Remark}

\newtheorem{defn}{Definition}
\begin{document}

\title{Constrained Optimal Stopping, Liquidity and Effort}
\author{David Hobson\thanks{University of Warwick, Coventry, CV4 7AL, d.hobson@warwick.ac.uk} \and Matthew Zeng\thanks{University of Warwick, Coventry, CV4 7AL, m.zeng@warwick.ac.uk}}
\date{\today}
\maketitle

\begin{abstract}
In a classical optimal stopping problem in continuous time, the agent can choose any stopping time without constraint. Dupuis and Wang (Optimal stopping with random intervention times, Advances in Applied Probability, 34, 141--157, 2002) introduced a constraint on the class of admissible  stopping times which was that they had to take values in the set of event times of an exogenous, time-homogeneous Poisson process. This can be thought of as a model of finite liquidity. In this article we extend the analysis of Dupuis and Wang to allow the agent to choose the rate of the Poisson process. Choosing a higher rate leads to a higher cost. Even for a simple model for the stopped process and a simple call-style payoff, the problem leads to a rich range of optimal behaviours which depend on the form of the cost function. Often the agent accepts the first offer --- if they are not going to accept an offer then there is no point in putting in effort to generate offers, and thus there may be no offers to accept or decline ---  but for some set-ups this is not the case.
\end{abstract}

%\section{Optimal stopping with constraints}\label{sec:lambda}
\section{Introduction}
\label{sec:intro}
Optimal stopping problems are widespread in economics and finance (and other fields) where they are used to model asset sales, investment times and the exercise of American-style options. In typical applications an agent observes a stochastic process, %$Y=(Y_t)_{t \geq 0}$,
possibly representing the price of an asset, and chooses a stopping time in order to maximise the expected discounted value of a payoff which is contingent upon the process evaluated at that time.

Implicit in the classical version of the above problem is the idea that the agent can sell the asset (decide to invest, exercise the option) at any moment of their choosing, and for financial assets traded on an exchange this is a reasonable assumption. However, for other classes of assets, including those described as `real assets' by, for example, Dixit and Pindyck~\cite{DixitPindyck:94}, this assumption may be less plausible. Here we are motivated by an interpretation of the optimal stopping problem above in which an agent has an asset for sale, but can only complete the sale if they can find a buyer, and candidate buyers are only available at certain isolated instants of time.

In this work we model the arrival of candidate purchasers as the event times of a Poisson process. When a candidate purchaser arrives the agent can choose to sell to that purchaser, or not; if a sale occurs then the problem terminates, otherwise the candidate purchaser is lost, and the problem continues. If the Poisson process has a constant rate, then the analysis falls into the framework studied by Dupuis and Wang \cite{DupuisWang:03} and Lempa~\cite{Lempa:07}.

Dupuis and Wang \cite{DupuisWang:03} and Lempa~\cite{Lempa:07} discuss optimal stopping problems, but closely related is the work of Rogers and Zane~\cite{RogersZane:99} in the context of portfolio optimisation. Rogers and Zane consider an optimal investment portfolio problem under the hypothesis that the portfolio can only be rebalanced at event times of a Poisson process of constant rate, see also Pham and Tankov~\cite{PhamTankov:08} and Ang, Papanikolaou and Westerfield~\cite{AngPapanikolaouWesterfield:14}. The study of optimal stopping problems when the stopping times are constrained to be event times of an exogenous process is relatively unexplored, but Guo and Liu~\cite{GuoLiu:05} study a problem in which the aim is to maximise a payoff contingent upon the maximum of an exponential Brownian motion and Menaldi and Robin~\cite{MenaldiRobin:16} extend the analysis of Dupuis and Wang~\cite{DupuisWang:03} to consider non-exponential inter-arrival times. As a generalisation of optimal stopping, Liang and Wei~\cite{LiangWei:16} consider an optimal switching problem when the switching times are constrained to be event times of a Poisson process.

In this article we consider a more sophisticated model of optimal stopping under constraints in which the agent may expend effort in order to increase the frequency of the arrival times of candidate buyers. (Note that the problem remains an optimal stopping problem, since at each candidate sale opportunity the agent optimises between continuing and selling.) In our model the agent's instantaneous effort rate $E_t$ affects the instantaneous rate $\Lambda_t$ of the Poisson process, so that the candidate sale opportunities become the event times of an inhomogeneous Poisson process, where the agent chooses the rate. However, this effort is costly, and the agent incurs a cost per unit time which depends on the instantaneous effort rate. The objective of the agent is to maximise the expected discounted payoff net of the expected discounted costs. In particular, if $X=(X_t)_{t \geq 0}$ with $X_0=x$ is the asset price process, $g$ is the payoff function, $\beta$ is the discount factor, $E = (E_t)_{t \geq 0}$ is the chosen effort process, $\Lambda = (\Lambda_t)_{t \geq 0}$ given by $\Lambda_t = \Psi(E_t)$ is the instantaneous rate of the Poisson process, $C_E$ is the cost function so that the cost incurred per unit time is $C_E(E_t)$, and $T_\Lambda$ is the set of event times of a Poisson process, rate $\Lambda$, then the objective of the agent is to maximise the objective function
\begin{equation}
\label{eq:mainproblem}
\E^x \left[ e^{- \beta \tau} g(X_\tau) - \int_0^\tau e^{-\beta s} C_E(E_s) ds \right]
\end{equation}
over admissible effort processes $E$ and $T_\Lambda$-valued stopping times $\tau$.
Our goal is to solve for the value function, the optimal stopping time and the optimal effort, as represented by the optimal control process $E$.
In fact, typically it is possible to use the rate of the Poisson process as the control variable by setting $C(\Lambda_t) = C_E(E_t) = C_E \circ \Psi^{-1} (\Lambda_t)$. In the context of the problem it is natural to assume that $\Psi$ and $C_E$ are increasing functions, so that $\Psi^{-1}$ exists, and $C$ is increasing.

Our focus is on the case where $X$ is an exponential Brownian motion, but the general case of a regular, time-homogeneous diffusion can be reduced to this case at the expense of slightly more complicated technical conditions. See Lempa~\cite{Lempa:07} for a discussion in the constant arrival rate case. We begin by rigorously stating the form of the problem we will study. Then we proceed to solve for the effort process and stopping rule in \eqref{eq:mainproblem}. It turns out that there are two distinctive cases depending on the shape of $C$ or more precisely on the finiteness or otherwise of $\lim_{\lambda \uparrow \infty} \frac{C(\lambda)}{\lambda}$. Note that it is not clear {\em a priori} what shape $C= C_E \circ \Psi^{-1}$ should take, beyond the fact that it is increasing. Generally one might expect an increasing marginal cost of effort and a law of diminishing returns to effort which would correspond to convex $C_E$, concave $\Psi$ and convex $C$. But a partial reverse is also conceivable: effort expended below a threshold has little impact, and it is only once effort has reached a critical threshold that extra effort readily yields further stopping opportunities; in this case $\Psi$ would be convex and $C$ might be concave.

One outcome of our analysis is that the agent exerts effort to create a positive stopping rate only if they are in the region where stopping is optimal. Outside this region, they typically exert no effort, and there are no stopping opportunities. Typically therefore, (although we give a counterexample in an untypical case) the agent stops at the first occasion where stopping is possible and the optimal stopping element of the problem is trivial.

\section{The set-up}
We work on a filtered probability space $(\Omega, \sF, \Prob, \F = (\sF_t)_{t \geq 0})$ which satisfies the usual conditions and which supports a Brownian motion and an independent Poisson process. On this space there is a regular, time-homogeneous diffusion process $X= (X_t)_{t \geq 0}$ driven by the Brownian motion. We will assume that $X$ is exponential Brownian motion with volatility $\sigma$ and drift $\mu$ and has initial value $x$; then
\[ dX_t = \sigma dW_t + \mu dt, \hspace{10mm} X_0 = x. \]

Here $\mu$ and $\sigma$ are constants with $\mu<\beta$. The agent has a perpetual option with increaing payoff $g: \R_+ \mapsto \R_+$ of linear growth. In our examples $g$ is an American call: $g(x)=(x-K)_+$. Then, in the classical setting, the problem of the agent would be to maximise $\E[e^{-\beta \tau} g(Y_\tau)]$ over stopping times $\tau$. Note that the linear growth condition, together with $\mu < \beta$, is sufficient to ensure that this classical problem is well-posed.

We want to introduce finite liquidity into this problem, in the sense that we want to incorporate the phenomena that in order to sell the agent needs to find a buyer, and such buyers are in limited supply. In the simplest case buyers might arrive at event times of a time-homogeneous Poisson process with rate $\lambda$, and then at each event time of the Poisson process the agent faces a choice of whether to sell to this buyer at this moment or not; if yes then the sale occurs and the optimal stopping problem terminates, if no then the buyer is irreversibly lost, and the optimal stopping problem continues. We want to augment this problem to allow the agent to expend effort (via networking, research or advertising) in order to increase the flow of buyers. There is a cost of searching in this way --- the higher the effort the higher the rate of candidate stopping times but also the higher the search costs. Note that once the asset is sold, effort expended on searching ceases, and search costs thereafter are zero by fiat.

Let $\sA_E$ be the set of admissible effort processes. We assume that $E \in \sA_E$ if $E = (E_t)_{t \geq 0}$ is an adapted process such that $E_t  \in I_{E}$ for all $t \in [0,\infty)$ where $I_{E} \subset \R_+$ is an interval which is independent of time. Then, since $\Lambda_t = \Psi(E_t)$ we find $E \in \sA_E$ if and only if $\Lambda \in \sA$ where $\Lambda \in \sA$ if $\Lambda$ is adapted and $\Lambda_t \in I$ for all $t$ where $I = \Psi(I_{E})$. Note that $I$ is an interval in $\R_+$, and we take the lower and upper endpoints to be $\underline{\lambda}$ and $\overline{\lambda}$ respectively.

Recall that $T_\Lambda$ is the set of event times of an inhomogeneous Poisson process with rate $\Lambda$. Then $T^\Lambda = \{ T^\Lambda_1, T^\Lambda_2, \ldots \}$ where $0<T^\Lambda_1$ and $T^{\Lambda}_n < T^{\Lambda}_{n+1}$ almost surely. Let $\sT(T_\Lambda)$ be the set of $T_\Lambda$-valued stopping times and let $\sA$ be the set of admissible rate functions. Then, after a change of independent variable the problem is to find
\begin{equation}
\label{eq:Hdef}
 H(x) = \sup_{\Lambda \in \sA} \sup_{\tau \in \sT(T_\Lambda)} \E^x \left[ e^{- \beta \tau} g(X_\tau) - \int_0^\tau e^{- \beta s} C(\Lambda_s) ds \right],
\end{equation}
together with the optimal rate function $\Lambda^* = (\Lambda^*_t)_{t \geq 0}$ and optimal stopping rule $\tau^* \in \sT(T_\Lambda)$.

In addition to the set of admissible controls, we also consider the subset of integrable controls $\sI \subseteq \sA$ where $\Lambda \in \sI = \sI(I,C)$ is an adapted process with $\Lambda_t \in I$ for which $\E^x[\int_0^\infty e^{- \beta s} C(\Lambda_s) ds ] < \infty$. As mentioned above we have that $\E^x \left[ e^{- \beta \tau} g(X_\tau) \right] < \infty$ for any admissible $\Lambda$ and any stopping rule, and hence there is no loss of generality in restricting the search for the optimal rate function to the set of integrable controls.

The stopping rule is easily identified in feedback form. Let $T^0_\Lambda = T_\Lambda \cup \{ 0 \}$ and let $H^0$ be the value of the problem {\em conditional on there being a buyer available at time 0}, so that
\[ H^0(x) = \sup_{\Lambda \in \sA} \sup_{\tau \in \sT(T^0_\Lambda)} \E^x \left[ e^{- \beta \tau} g(X_\tau) - \int_0^\tau e^{- \beta s} C(\Lambda_s) ds  \right]. \]
Then, it is optimal to stop immediately if and only if the value of stopping is at least as large as the value of continuing and
\[ H^0(x) = \max \{ g(x), H(x) \} . \]

It follows that if $\Lambda = (\Lambda_t)_{t \geq 0}$ is a fixed admissible rate process, and if $H^0_\Lambda$ and $H_\Lambda$ denote the respective value functions then, writing $T_1 = T^\Lambda_1$ for the first event time of the Poisson process with rate $\Lambda$,
\begin{eqnarray*}
H_\Lambda(x) & = & \sup_{\tau \in \sT(T_\Lambda)} \E^x \left[ e^{- \beta \tau} g(X_\tau) - \int_0^\tau e^{- \beta s} C(\Lambda_s) ds \right] \\
& = & \sup_{\tau \in \sT(T_\Lambda)} \E^x \left[ e^{- \beta T_1} \E \left[ \left.  e^{-\beta (\tau - T_1)} g(X_\tau) - \int_{T_1}^{\tau} e^{- \beta (s - T_1)} C(\Lambda s) ds  \right| \sF_{T_1} \right]  - \int_0^{T_1} e^{-\beta s} C(\Lambda_s) ds \right]  \\
%& = & \E^x \left[ e^{-\beta T_1} H^0(X_{T_1}) - \int_0^{T_1} e^{-\beta s} C(\Lambda_s) \right] \\
& = & \E^x \left[ e^{-\beta T_1} H^0_\Lambda(X_{T_1}) - \int_0^{\infty} I_{ \{ s < T_1 \} }e^{-\beta s} C(\Lambda_s) ds \right] \\
& = & \E^x \left[ \int_0^\infty \Lambda_s e^{- \int_0^s \Lambda_u du} e^{-\beta s} H^0_\Lambda(X_{s}) ds - \int_0^{\infty} e^{- \int_0^s \Lambda_u du} e^{-\beta s} C(\Lambda_s) ds \right] \\
& = & \E^x \left[ \int_0^\infty e^{- (\beta s + \int_0^s \Lambda_u du)} (\Lambda_s H^0_\Lambda(X_{s}) ds - C(\Lambda_s)) ds \right] .
\end{eqnarray*}
Taking a supremum over admissible rate processes $\Lambda \in \sA$ we find
\[ H(x) = \sup_{\Lambda \in \sA} \E^x \left[ \int_0^\infty  e^{- \int_0^t (\beta + \Lambda_s) ds} \left( \Lambda_t H_\Lambda^0(X_t) - C(\Lambda_t) \right) dt \right] ,
\]
and this is the problem we aim to solve. Writing $\Lambda^*$ for the optimal rate process we expect $H$ to solve
\[ H(x) = \E^x \left[ \int_0^\infty  e^{- \int_0^t (\beta + \Lambda^*_s) ds} \left( \Lambda^*_t \{ g(X_t) \vee H(X_t)\} - C(\Lambda^*_t) \right) dt \right] .
\]

\subsection{Some results for classical problems}
\label{ssec:classical}
For future reference we record some results for classical problems in which agents can stop at any instant.

First, let $\sT([0,\infty))$ be the set of all stopping times and define
\[ w_K(x) := \sup_{\tau \in \sT([0,\infty))} \E^x[ e^{-\beta \tau}(X_\tau - K)_+ ]. \]
(Imagine a standard, perpetual, American-style call option with strike $K$, though valuation is not taking place under the equivalent martingale measure.)
Classical arguments (McKean~\cite{McKean:65}, Peskir and Shiryaev~\cite{PeskirShiryaev:06}) give that $0 < w_K < x$ (the upper bound holds since we are assuming $\beta > \mu$) and that there exists a constant $L = \frac{\theta}{\theta - 1} K$ where $\theta = \left( \frac{1}{2} - \frac{\mu}{\sigma^2} \right) + \sqrt{\left( \frac{1}{2} - \frac{\mu}{\sigma^2} \right)^2 + \frac{2 \beta}{\sigma^2}} $ such that
\[ w_K(x) = \left\{ \begin{array}{ll} (x-K)_+, & x> L; \\
                                       (L-K) L^{-\theta} x^\theta, & 0 < x \leq L .\end{array} \right. \]
For future reference set $\phi = \left( \frac{1}{2} - \frac{\mu}{\sigma^2} \right) - \sqrt{\left( \frac{1}{2} - \frac{\mu}{\sigma^2} \right)^2 + \frac{2 \beta}{\sigma^2}}$. Then $\phi<0<1<\theta$ and $\theta$ and $\phi$ are the roots of $Q_0=0$ where $Q_\lambda(\psi)=\frac{1}{2} \sigma^2 \psi(\psi-1) + \mu \psi - (\beta+ \lambda)$.

Second, define
\begin{equation} w_{K,\epsilon,\delta}(x) = \sup_{\tau \in \Tau([0,\infty))} \E^x \left[ e^{-\beta \tau} \{ (X_\tau - K)_+ - \epsilon \} - \delta \int_0^\tau e^{-\beta s}  ds \right]. \label{eq:wKJD}
\end{equation}
(Imagine a perpetual, American-style call option with strike $K$, in which the agent pays a fee or transaction cost $\epsilon$ to exercise the option, and pays a running cost $\delta$ per unit time until the option is exercised.) Note that $w_{K,0,0} \equiv w_K$. It turns out that there are two cases. In the first case when $\epsilon \geq \delta/\beta$, when $X$ is small it is more cost effective to pay the running cost indefinitely than to pay the exercise fee. We find
\[ w_{K,\epsilon,\delta}(x) = w^{K+\epsilon - \delta/\beta}(x) - \delta/\beta . \]
In the second case when $\epsilon< \delta/\beta$, when $X$ is small it is cost-effective to stop immediately, even though the payoff is zero, because paying the fee is cheaper than paying the running cost indefinitely. In this case we find that $w = w_{K,\epsilon,\delta}$ and a pair of thresholds $l^* = l^*(K,\epsilon,\delta)$ and  $L^*=L^*(K,\epsilon,\delta)$ with $0 < l^* < K + \epsilon < L^*$ satisfy the variational problem
\[  \{ \mbox{$w$ is $C^1$; $w = -\epsilon$ on $(0,l^*)$; $\sL w - \beta w = \delta$ on $(l^*,L^*)$; $w=x-K - \epsilon$ on $(L^*, \infty)$ } \} \]

Returning to our problem with limited stopping opportunities, one immediate observation is that $H(x) \leq w_K(x)$. Conversely, if $\Lambda \equiv 0$ is admissible then $H(x) \geq - \frac{C(0)}{\beta}$.

\section{Heuristics}
From the Markovian structure of the problem we expect that the (unknown) value function $H$ and optimal rate function $\Lambda^*$ are time-homogeneous functions of the asset price only.

Let $M^\Lambda = (M^\Lambda_t)_{t \geq 0}$ be given by
\[ M^\Lambda_t = e^{-\int_0^t ( \beta + \Lambda_s) ds} H(X_t) + \int_0^t e^{-\int_0^u ( \beta + \Lambda_s) ds} \left[ \Lambda_u H^0(X_u) - C(\Lambda_u) \right] du, \]
and let $\sL^X$ denote the generator of $X$ so that $\sL^X f = \frac{\sigma^2}{2} f'' + \mu f'$. Assume that the value function under the optimal strategy $H$ is $C^2$. Then, by It\^{o}'s formula,
\[ dM^{\Lambda}_t = e^{-\int_0^t ( \beta + \Lambda_s) ds}\left\{ \left( \sL^X H(X_t) - (\beta + \Lambda_t) H(X_t) + \Lambda_t (H^0(X_t) - C(\Lambda_t)) \right) dt + \sigma X_t H'(X_t) dW_t \right\}. \]
We expect that $M^\Lambda$ is a super-martingale for any choice of $\Lambda$, and a martingale for the optimal choice. Thus we expect
\[  \sL^X H(X_t) - \beta H(X_t) - \inf_{\Lambda_t} \left\{ C(\Lambda_t) - \Lambda_t [ H^0(X_t) - H(X_t) ]  \right\} = 0. \]
Let $\tilde{C}: \R_+ \mapsto \R$ be the concave conjugate of $C$ so that $\tilde{C}(z) = \inf_{\lambda \geq 0} \{C(\lambda) - \lambda z \}$.
Then we find that $H$ solves
\begin{equation} \sL^X H - \beta H - \tilde{C}(H^0 - H) = 0, \label{eq:Hheuristic} \end{equation}
and a best choice of rate function is
$ \Lambda^*_t = \Lambda^*(X_t)$ where
\begin{equation} \Lambda^*(x) = \Theta(H^0(x)-H(x)) \label{eq:controlq}
\end{equation}
and $\Theta(z) = \mbox{arginf}_\lambda \{ C(\lambda) - \lambda z \}$.
Note that $H^0-H = (g-H)_+$ and that \eqref{eq:Hheuristic} is a second order differential equation and will have multiple solutions. The boundary behaviour near zero and infinity will determine which solution fits the optimal stopping problem.

\subsection{First Example: Quadratic cost functions}
\label{ssec:quadratic}
Suppose $g(x)=(x-K)_+$ for fixed $K>0$. Using terminology from the study of American options and optimal stopping we say that if $X_t > K$ then the process is in-the-money, if $X_t<K$ then the process is out-of-the-money and the region in the domain of $X$ where $\Lambda^*(X)$ is zero is the continuation region $\sC$, and $\sS := \R^+ \setminus \sC$ is the selling region.

Suppose the range of possible values for the rate process is $I=[0,\infty)$ and consider a quadratic cost function $C(\lambda) = a + b \lambda + c\frac{\lambda^2}{2}$ with $a \geq 0$, $b \geq 0$ and $c>0$. Then $\tilde{C}(z) = a - \frac{[(z-b)_+]^2}{2c}$.

Consider first the behaviour of the value function near zero. If $a=0$ then $C(0)=0$, and when $X$ is close to zero the agent may choose not to search for buyers, a strategy which incurs zero cost. There is little chance of the process ever being in-the-money, but nonethelesss the agent delays sale indefinitely. We expect that the continuation region is $(0,L^*)$ for some threshold $L^*$.

Now suppose $a>0$. Now there is a cost to delaying the sale, even when $\Lambda = 0$. If $X$ is small then it is preferable to sell the asset even though the process is out-of-the-money, because in our problem there are no search costs once the asset is sold. In this case we expect the agent to search for buyers when $X$ is small, in order to reduce further costs. Then the continuation region will be $(\ell^*, L^*)$ for some $0<\ell^* < K < L^* < \infty$.

Consider now the behaviour for large $x$. In this case we can look for an expansion for the solution of \eqref{eq:Hheuristic} of the form
\[ H(x) = A_1 x + A_{1/2} \sqrt{x} + A_0 + O(x^{-1/2}) \]
for constants $A_1$, $A_{1/2}$ and $A_0$ to be determined. Using the fact that $H(x) \leq w_K(x)$ so that $H$ is of at most linear growth we find
\begin{equation}
H(x) = x - \sqrt{2c(\beta - \mu)} \sqrt{x} - \left\{ K + b - c\left[ \beta - \frac{\mu}{2} + \frac{\sigma^2}{8} \right] \right\} + \ldots \label{eq:Gexpansion}
\end{equation}
Numerical results (see Figure~\ref{fig:valuefun}) show that this expansion is very accurate for large $x$.

\subsubsection{Purely quadratic cost: $a=0=b$}
\label{sssec:pure}
In this case we expect that the continuation region is $(0,L^*)$ for a threshold level $L^*$ to be determined. For a general threshold $L$, and writing $H_L$ for the solution to \eqref{eq:Hheuristic} with $H(0)=0$ and $H(L)=L-K$ we find that $H_L$ solves
\begin{equation}
\sL^X h - \beta h = \frac{1}{2c} (\{g-h\}_+)^2,
\label{eq:ODEa=0}
\end{equation}
and then that $H_L(x) = \frac{L-K}{L^\theta}x^\theta$ on $x \leq L$. On $(L,\infty)$, $H_L$ solves \eqref{eq:ODEa=0} subject to $H_L(L) = (L-K)$ and $H_L'(L) = \theta \frac{L-K}{L}$. This procedure gives us a family $(H_L)_{L \geq K}$ of potential value functions, each of which is $C^1$. Finally we can determine the threshold level $L$ we need by choosing the value $L^*$ for which $H_{L^*}$ has linear growth at infinity.
%, or equivalently has expansion \eqref{eq:Gexpansion} for large $x$.

\begin{figure*}[!ht]
    \centering
    \begin{subfigure}[t]{0.45\textwidth}
        \includegraphics[width=1\textwidth]{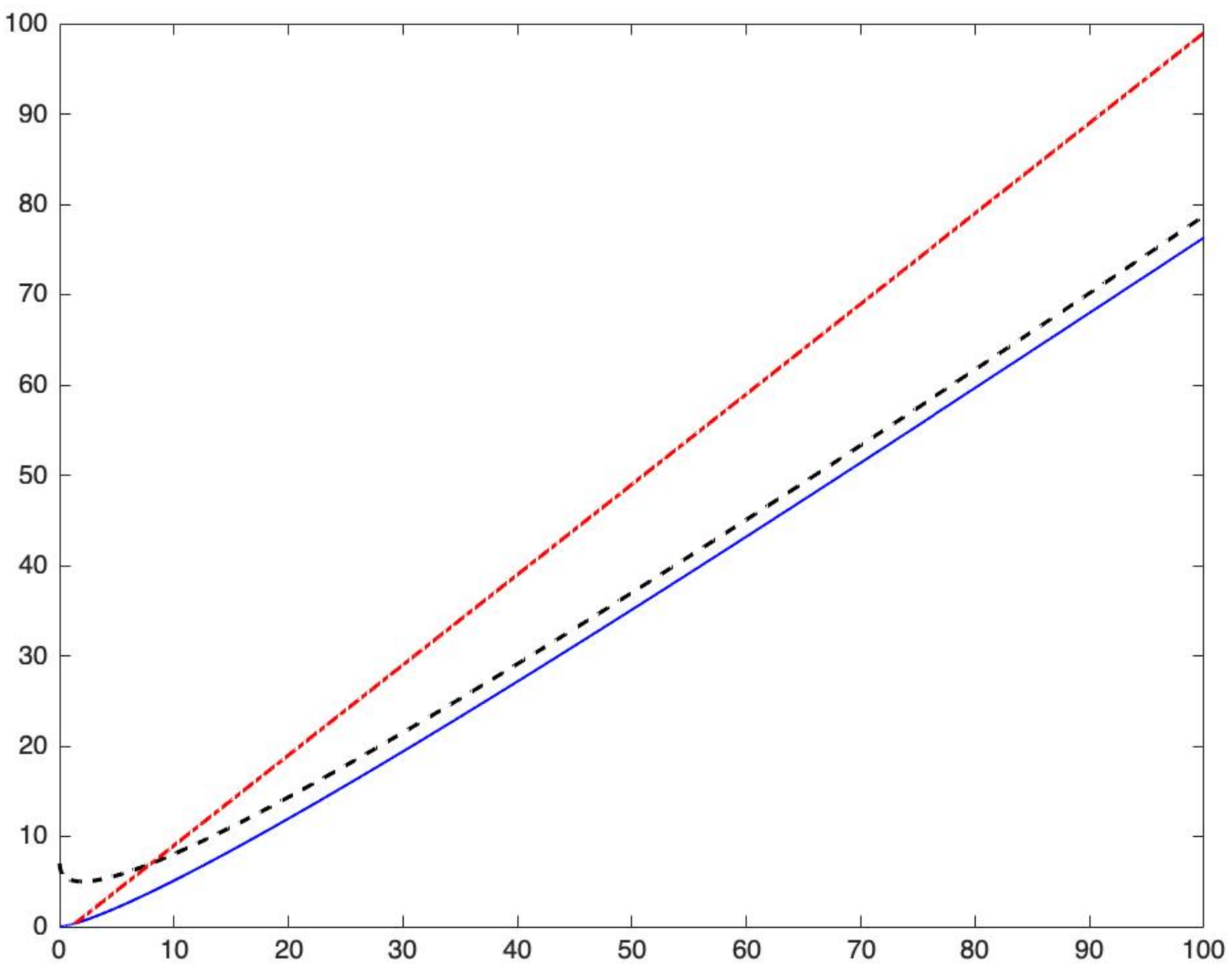}
        \caption{For large x}
    \end{subfigure}%
    ~
    \begin{subfigure}[t]{0.45\textwidth}
        \centering
        \includegraphics[width=1\textwidth]{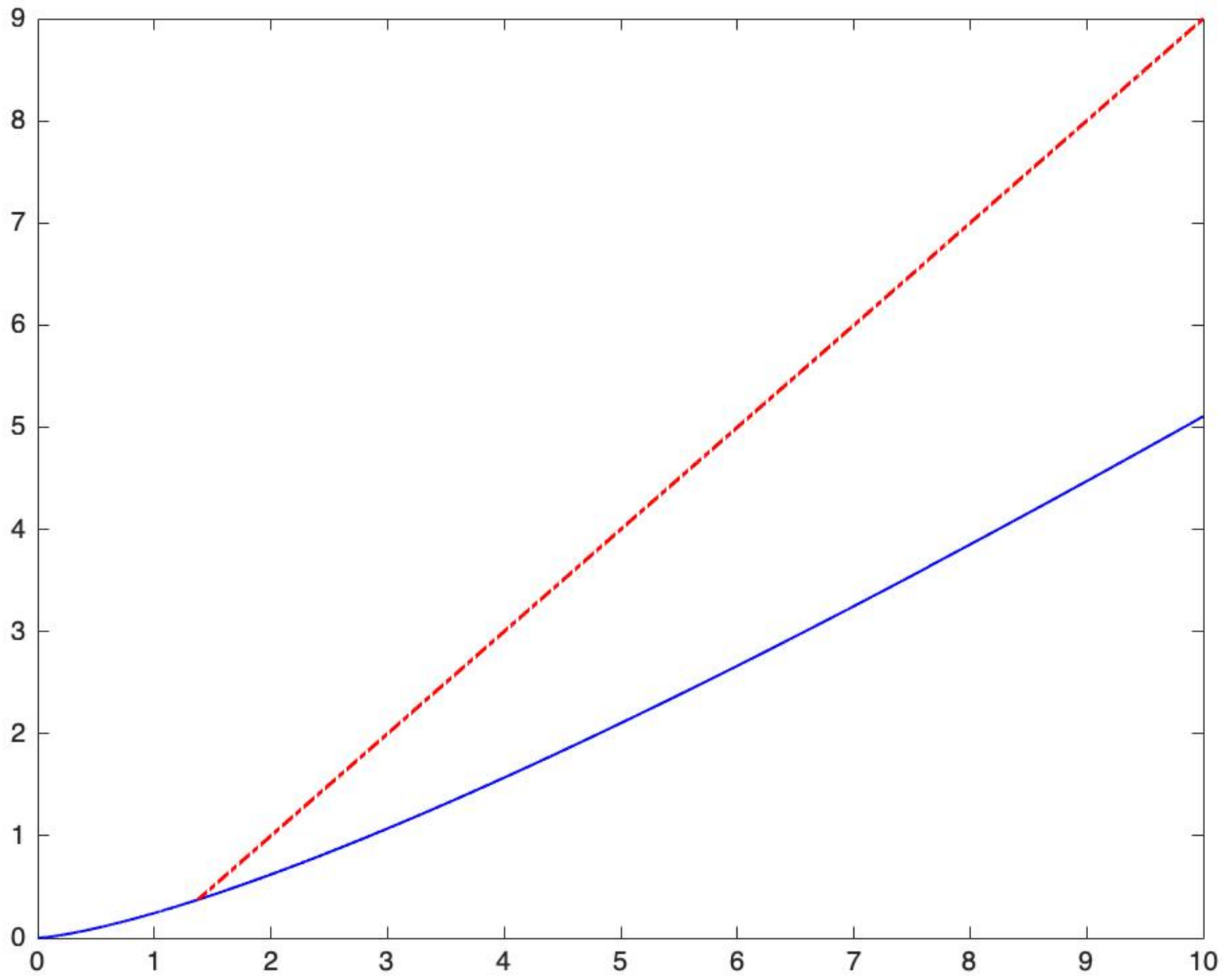}
        \caption{For small x}
    \end{subfigure}
    \caption{$(\beta,\mu,\sigma,K,a,b,c) =(5, 3 ,2,1,0,0,2)$. In both sub-figures the solid curved line represents $H_{L^*}$; the straight line represents $g \vee H_{L^*}$ on $\{ x : g(x) \geq H_{L^*}(x) \}$ and the dashed line in the left sub-figure is the expansion for $H$ in \eqref{eq:Gexpansion}. The optimal threshold is seen in the right sub-figure to be at $L^*= 1.35$. }
    \label{fig:valuefun}
 \end{figure*}

The linear growth solution $H_{L^*}$ is shown in Figure~\ref{fig:valuefun}, both for large $x$ and for moderate $x$. From Figure~\ref{fig:valuefun}(b) we see that the continuation region is $\sC = (0,1.35)$ and that the stopping region $\sS = [1.35,\infty)$. We also see that the expansion for $H$ given in \eqref{eq:Gexpansion} gives a good approximation of our numerical solution for large $x$.

  \begin{figure}[H] \center
\includegraphics [width=0.45\linewidth] {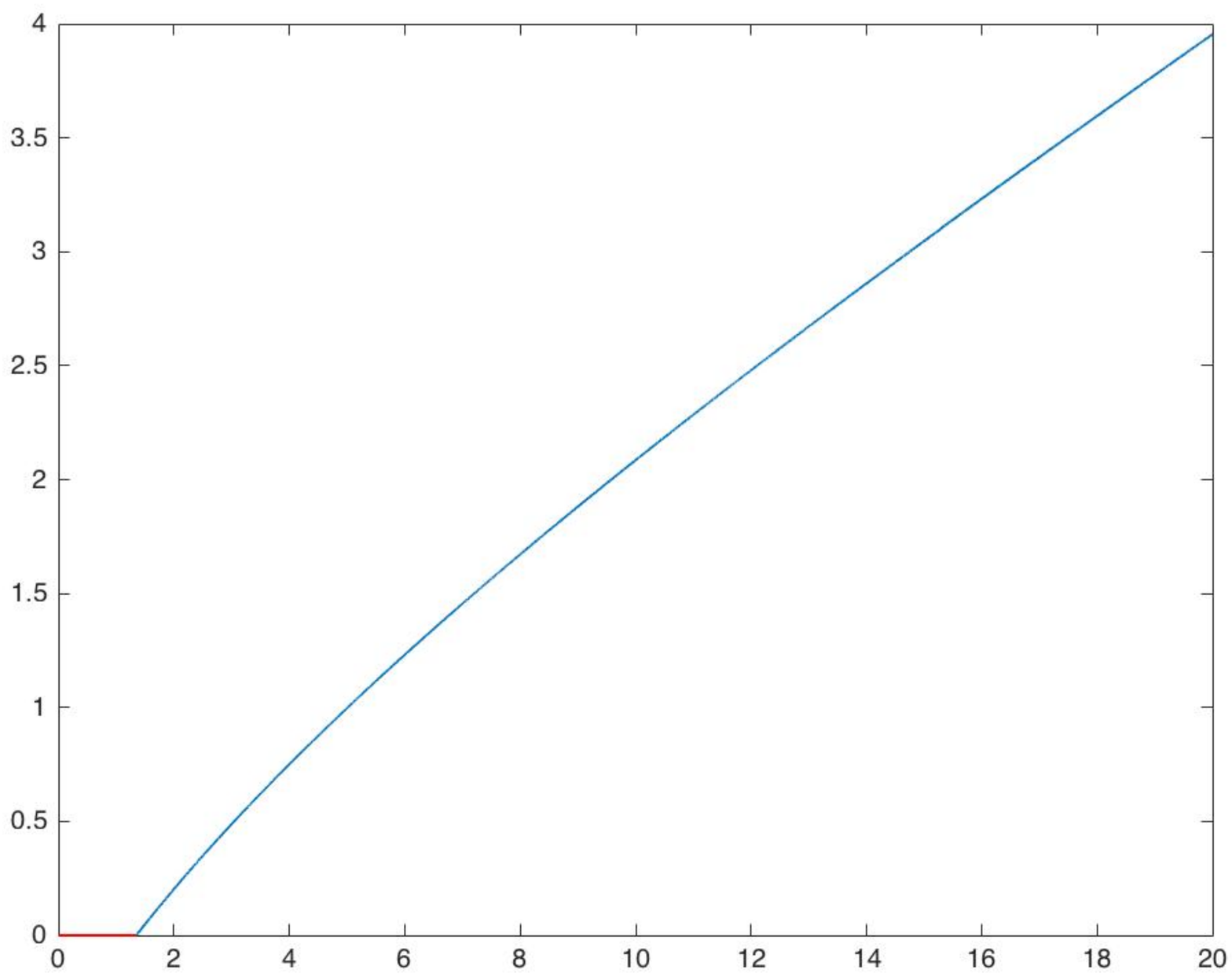}
\caption{ $(\beta,\mu,\sigma,K,a,b,c) =(5, 3 ,2,1,0,0,2)$; this figure plots the optimal control $\Lambda^*$ given by \eqref{eq:controlq} as a function of wealth level $x$.   }
\label{fig:lambda}
\end{figure}

Figure~\ref{fig:lambda} shows the optimal control. We see that $\Lambda^*$ is zero on the continuation region $\sC=(0,L)$ and that $\Lambda^*$ is increasing and concave on the stopping region $\sS=[L,\infty)$.  The agent behaves rationally in the sense that on the continuation region where continuing is worth more than stopping, the agent is unwilling to stop and this is reflected by the minimal efforts spent on searching (i.e. $\Lambda^*(x)=0, \forall x \in  \sC$); similarly, on the stopping region, stopping is getting more and more valuable relative to continuing as the price process gets deeper in-the-money, and the agent is incentivised to spend more effort on searching for stopping opportunities.

We discuss the cases of $a>0$ and $b>0$ in Section~\ref{sec:examples}.

\section{Verification}

In this section we show that the heuristics are correct, and that the value to the stochastic problem is given by the appropriate solution of the differential equation. Although the details are different, the structure of the proof follows Dupuis and Wang~\cite{DupuisWang:03}.

%Suppose $C$ is increasing and convex, suppose $\sA$ is the set of all adapted, non-negative processes and suppose that .
Suppose, as throughout, that $X$ is exponential Brownian motion with $\mu <\beta$ and $g$ is of linear growth.

\begin{defn}
$(\tau, \Lambda)$ is admissible if $\Lambda$ is a non-negative, $I$-valued, adapted process and $\tau \in \sT(T^\Lambda)$. %{\bf Do we want $\int_0^t \Lambda_s ds < \infty$ almost surely for all $t$? Or }
\end{defn}

Note that a consequence of the definition is that we insist that $\tau \leq T^{\Lambda}_\infty := \lim_n T^\Lambda_n$. Moreover, we may have $T_k = \infty$: in this case we may take $\tau = \infty$, whence we have $e^{-\beta \tau} g(X_\tau) = 0$ noting that $\lim_{t \uparrow \infty} e^{-\beta t} g(X_t)=0$ almost surely.

\begin{defn}
$(\tau, \Lambda)$ is integrable if $(\tau, \Lambda)$ is admissible and $\E[ \int_0^\tau e^{- \beta s} C(\Lambda_s) ds ] < \infty$.
\end{defn}

Clearly, if $(\tau,\Lambda)$ is integrable, then $(T^\Lambda_1, \Lambda)$ is integrable.

\begin{lem}
\label{lem:Gcomparison}
Let $G$ be an increasing, convex solution to
\begin{equation}  \sL^X G - \beta G - \tilde{C}((g - G)_+) = 0, \label{eq:Gdef} \end{equation}
and suppose that $G$ is of at most linear growth. Set $G^0 = G \vee g$.

Then for any integrable, admissible strategy $(\tau, \Lambda)$,
%writing $T= T^\Lambda_1$ the first event time of the Poisson process
%\begin{eqnarray}
\begin{equation}
G(x) % =  \E^x \left[ e^{-\beta T} G^0(X_T) I_{ \{ T < \infty \} } - \int_0^T e^{- \beta s} \left( (g(X_s)-G(X_s))_+ \Lambda_s + \tilde{C}((g(X_s) - H(X_s))_+) \right) ds \right] \label{eq:H=} \\
 \geq  \E^x \left[ e^{-\beta T^\Lambda_1} G^0(X_{T^\Lambda_1}) I_{ \{ T^\Lambda_1 < \infty \} } - \int_0^{T^\Lambda_1} e^{- \beta s} C(\Lambda_s) ds \right].  \label{eq:Hleq}
%\end{eqnarray}
\end{equation}
%with equality for $\Lambda_s = (C')^{-1} (g(X_s) - G(X_s))_+)$, provided this quantity is well defined.
\end{lem}

\begin{proof}
Since $g$ and $G$ are of linear growth we may assume $G^0(x) \leq \kappa_0 + \kappa_1 x$ for some constants $\kappa_i \in (0,\infty)$.

Let $Z_t = e^{- \beta t - \int_0^t \Lambda_s ds}G(X_t) - \int_0^t e^{-\beta s  - \int_0^s \Lambda_u du} F_s ds$ where
\[ F_s = F(g(X_s), G(X_s), \Lambda_s) := (g(X_s) -G(X_s))_+ \Lambda_s + \tilde{C}((g(X_s) -G(X_s))_+) \leq C(\Lambda_s). \]
Then, using the definition of $G$
\begin{eqnarray*}
dZ_t & = & e^{- \beta t - \int_0^t \Lambda_s ds} \left\{ -(\beta + \Lambda_t) G + \sL^X G - (g-G)_+ \Lambda_t - \tilde{C}((g-G)_+) \right\} dt + dN_t \\
& = & e^{- \beta t - \int_0^t \Lambda_s ds} \left\{ - \Lambda_t [G + (g-G)_+] \right\} dt + dN_t \\
& = & - e^{- \beta t - \int_0^t \Lambda_s ds}  \Lambda_t G^0(X_t) dt + dN_t
\end{eqnarray*}
where $N_t = \int_0^t e^{- \beta s - \int_0^s \Lambda_u du} \sigma X_s G'(X_s)dW_s$. Our hypotheses on $G$ allow us to conclude that $N=(N_t)_{t \geq 0}$ is a martingale.

It follows that $Z_0 = \E[Z_t + \int_0^t e^{- \beta s - \int_0^s \Lambda_u du} \Lambda_s G^0(X_s) ds]$ or equivalently
\begin{eqnarray} G(x) & = & \E^x \left[ e^{- \beta t - \int_0^t \Lambda_s ds} G(X_t) + \int_0^t e^{- \beta s - \int_0^s \Lambda_u du} \left( \Lambda_s(g(X_s) \vee G(X_s)) - F_s \right) ds \right] \nonumber \\
& \geq & \E^x \left[ e^{- \beta t - \int_0^t \Lambda_s ds} G(X_t)+  \int_0^t e^{- \beta s - \int_0^s \Lambda_u du} \left( \Lambda_s G^0(X_s) - C( \Lambda_s) \right) ds \right].
\label{eq:Ht}
\end{eqnarray}
Since $X$ is geometric Brownian motion and $\beta > \mu$ we have that $X^{\beta,*} := \sup_{u \geq 0} \{ e^{- \beta u} X_u \}$ is in $L^1$. Then
\begin{eqnarray*}
e^{- \beta t - \int_0^t \Lambda_s ds} G(X_t) &\leq&  \kappa_0 + \kappa_1 X^{\beta,*}, \\
 \int_0^t e^{- \beta s - \int_0^s \Lambda_u du} \Lambda_s G^0(X_s) ds & \leq & (\kappa_0 + \kappa_1 X^{\beta,*}) \int_0^t \Lambda_s  e^{- \int_0^s \Lambda_u du} ds \leq\kappa_0 + \kappa_1 X^{\beta,*}, \\
\int_0^t e^{- \beta s - \int_0^s \Lambda_u du} C(\Lambda_s) ds &\leq&  \int_0^\infty e^{- \beta s - \int_0^s \Lambda_u du} C(\Lambda_s) ds,
\end{eqnarray*}
and, since $(T^\Lambda_1, \Lambda)$ is integrable by hypothesis,
\[ \E \left[ \int_0^\infty e^{- \beta s - \int_0^s \Lambda_u du} C(\Lambda_s) ds \right] = \E \left[ \int_0^{T^\Lambda_1} e^{- \beta s} C(\Lambda_s) ds \right] < \infty. \]
Then Dominated Convergence, together with the fact that $e^{-\beta t}X_t \rightarrow 0$ gives \eqref{eq:Hleq}.

\end{proof}

\begin{lem}
Let $(\tau, \Lambda)$ be an integrable strategy. % and let $T_n = T^\Lambda_n$ be the $n$th event time of the Poisson process rate $\Lambda$ (with $T_0=0$).
Define $Y = (Y_n)_{n \geq 0}$ by
\[ Y_n = e^{-\beta (T^\Lambda_n \wedge \tau)} G^0(X_{T^\Lambda_n \wedge \tau}) I_{ \{ T^\Lambda_n \wedge \tau < \infty \}} - \int_0^{T^\Lambda_n \wedge \tau} e^{- \beta s} C(\Lambda_s) ds \]
where $T^\Lambda_0=0$. Define $\sG_n = \sF_{T^\Lambda_n}$ and set $\G = (\sG_n)_{n \geq 0}$.

Then $Y$ is a uniformly integrable $(\sG_n)_{n \geq 0}$-supermartingale.
\label{lem:Ysupermg}
\end{lem}

\begin{proof}
We have
\[ |Y_n| \leq \kappa_0 + \kappa_1 X^{\beta,*} + \int_0^\tau e^{-\beta s} C(\Lambda_s) ds \in L^1. \]
Moreover, on $T^\Lambda_{n-1}<\infty$ and $\tau > T^\Lambda_{n-1}$, writing $\tilde{T}$ as shorthand for $T^\Lambda_n-T^\Lambda_{n-1}$ and using $\tau \geq T^\Lambda_n$ and Lemma~\ref{lem:Gcomparison} for the crucial first inequality,
\begin{eqnarray*}
\E[Y_n | \sG_{n-1}] & = & e^{- \beta T^\Lambda_{n-1}} \E \left[ \left. e^{- \beta \tilde{T}} G^0(X_{T^\Lambda_n}) I_{ \{ T^\Lambda_n < \infty \}} - \int_{T^\Lambda_{n-1}}^{T^\Lambda_n} e^{- \beta s} C(\Lambda_s) ds \right| \sG_{n-1} \right] - \int_0^{T^\Lambda_{n-1}} e^{- \beta s} C(\Lambda_s) ds \\
& \leq & e^{- \beta T^\Lambda_{n-1}} G(X_{T^\Lambda_{n-1}}) - \int_0^{T^\Lambda_{n-1}} e^{- \beta s} C(\Lambda_s) ds \\
& \leq & e^{- \beta T^\Lambda_{n-1}} G^0(X_{T^\Lambda_{n-1}}) - \int_0^{T^\Lambda_{n-1}} e^{- \beta s} C(\Lambda_s) ds = Y_{n-1}.
\end{eqnarray*}
\end{proof}

\begin{prop}
\label{prop:HleqG}
Let $G$ be an increasing, convex solution to \eqref{eq:Gdef} of at most linear growth. Then $H \leq G$.

%Suppose further that $C$ is strictly convex and $(C')^{-1}(0) =0$. Then $H=G$.
\end{prop}

\begin{proof} Let $(\tau, \Lambda)$ be any integrable strategy.

From Lemma~\ref{lem:Gcomparison} we have
\[ \E[Y_1] = \E \left[ e^{-\beta T^\Lambda_1} G^0(X_{T^\Lambda_1}) I_{\{ T^\Lambda_1 < \infty \} } - \int_0^{T^\Lambda_1} e^{-\beta s} C(\Lambda_s) ds \right] \leq G(x). \]

Moreover, since $Y$ is a uniformly integrable supermartingale,
\begin{eqnarray*} \E^x[Y_1] \geq \E^x[Y_\infty] & = & \E \left[ e^{-\beta \tau} G^0(X_\tau) I_{\{ \tau < \infty \} } - \int_0^\tau e^{-\beta s} C(\Lambda_s) ds \right] \\
& \geq & \E \left[ e^{-\beta \tau} g(X_\tau) I_{\{ \tau < \infty \} } - \int_0^\tau e^{-\beta s} C(\Lambda_s) ds \right].
\end{eqnarray*}
Taking a supremum over stopping times and rate processes we conclude that $H(x) \leq G(x)$.
\end{proof}

Our goal now is to show that $H=G$. We prove this result, first in the simplest case where the set of admissible rate processes is unrestricted (i.e. $\Lambda_t$ takes values in $I=[0,\infty)$ and the cost function $C$ is lower semi-continuous and convex, with $\lim_{\lambda \uparrow \infty} C(\lambda)/\lambda = \infty$). Then we argue that the same result holds true under weaker assumptions. Note that we allow for $\{ \lambda \in I : C(\lambda)=\infty \}$ to be non-empty, but our assumption that $C$ is lower semi-continuous means that if $\check{\lambda} = \inf \{ \lambda : C(\lambda)= \infty \}$ then $C(\check{\lambda}) = \lim_{\lambda \uparrow \check{\lambda}}C(\lambda)$.

\begin{thm}
Suppose $I = [0,\infty)$ and $C: I \mapsto [0,\infty]$ is increasing, convex and lower semi-continuous with $\lim_{\lambda \uparrow \infty} C(\lambda)/\lambda = \infty$.
Let $G$ be an increasing, convex solution to \eqref{eq:Gdef}
%\begin{equation}  \sL^X G - \beta G - \tilde{C}((g - G)_+) = 0, \label{eq:Gdef2} \end{equation}
of at most linear growth. Then $H=G$.
\label{thm:H=G1}
\end{thm}

\begin{proof}
Let $C'$ denote the right-derivative of $C$ and set $C'=\infty$ on $\{ \lambda : C(\lambda)=\infty \}$.
Since $C'$ is increasing it has a left-continuous inverse $D : \R_+ \mapsto \R_+$. In particular, $D(y) = \sup \{ \lambda \in [0,\infty): C'(\lambda)  < y \}$ with the convention that $D(y)=0$ if $C'(\lambda) \geq y$ on $(0,\infty)$. We note that our hypotheses mean that $D$ is well defined and finite on $(0,\infty)$ and we set $D(0)=0$.

Let $\hat{\Lambda} = (\hat{\Lambda}_s)_{s \geq 0}$ be given by $\hat{\Lambda}_s = D( (g(X_s) - G(X_s))_+)$. We will show that $\hat{\Lambda}$ is the optimal rate process.

Note first that there is equality in \eqref{eq:Ht}, and therefore in \eqref{eq:Hleq}, provided $F_s = F(g(X_s), G(X_s), \Lambda_s) = (g(X_s) -G(X_s))_+ \Lambda_s + \tilde{C}((g(X_s) -G(X_s))_+) = C(\Lambda_s)$. This is satisfied if $\Lambda_s = \hat{\Lambda}_s$. %{\bf References???}

Let $\sX_> = \{ x : g(x) > G(x) \}$ and let $\sX_{\leq} = \{ x : g(x) \leq G(x) \}$.
%Suppose $\Lambda^*_s = (C')^{-1}(g(X_s) - G(X_s))_+)$.
Then, under the hypothesis of the theorem, whilst $X_\cdot \in \sX_\leq$ we have that $\hat{\Lambda}_\cdot \equiv 0$. Hence (almost surely) $X_{T^{\hat{\Lambda}}_1} \in \sX_>$ and $G^0(X_{T^{\hat{\Lambda}}_1}) = g(X_{T^{\hat{\Lambda}}_1})$. Then, taking $T = T^{\hat{\Lambda}}_1$ we have from \eqref{eq:Hleq} that
\begin{eqnarray*}
 G(x) & = &\E \left[ e^{-\beta T} G^0(X_{T}) I_{\{ T < \infty \} } - \int_0^{T} e^{-\beta s} C(\Lambda_s) ds \right] \\
& = &\E \left[ e^{-\beta T} g(X_{T}) I_{\{ T < \infty \} } - \int_0^{T} e^{-\beta s} C(\Lambda_s) ds \right] \leq H(x)
\end{eqnarray*}
and hence, combining with Proposition~\ref{prop:HleqG}, $G=H$.
\end{proof}

\begin{cor}
$\Lambda^* = (\Lambda_s^*)_{s \geq 0}$ given by $\Lambda^*_s = D((g(X_s) - G(X_s))_+)$ is an optimal strategy, and $\tau^* = T^{\Lambda^*}_1$ is an optimal stopping rule.
\label{cor:Lambda*}
\end{cor}

Our goal now is to extend Theorem~\ref{thm:H=G1} to allow for more general admissibility sets and cost functions.

Let $c$ be a generic increasing, convex function $c : [0,\infty) \mapsto [0,\infty]$. If $c$ takes the value $+\infty$ on $(\check{\lambda},\infty)$ then we assume that $c(\check{\lambda}) = \lim_{\lambda \uparrow \infty} c(\lambda) = c(\check{\lambda})$, and set the right-derivative $c'$ equal to infinity on $(\check{\lambda},\infty)$ also. For such a $c$ define $D_c:[0,\infty) \mapsto [0,\infty]$ by $D_c(y) = \sup \{ \lambda \in (0,\infty): c'(y) < y \}$ again with the conventions that $D_c(y)=0$ if $c'(\lambda) \geq y$ on $(0,\infty)$ and $D_c(0)=0$. Note that $D_c(y) \leq \sup \{ y : c(y) < \infty \}$.

Let $I$ with endpoints $\{\underline{\lambda}, \overline{\lambda} \}$ be a subinterval of $[0,\infty)$ with the property that $I$ is closed on the left and closed on the right if $\overline{\lambda} < \infty$.

Let $\gamma : I \mapsto \R_+$ be an increasing function. Let $\breve{\gamma}$ be the largest convex minorant of $\gamma$ on $I$. The
define $\gamma^\dagger$ by $\gamma^\dagger(\lambda) = \gamma(\underline{\lambda})$ on $[0,\underline{\lambda})$ (if this interval is non-empty), $\gamma^\dagger(\lambda) = \breve{\gamma}(\lambda)$ on $[\underline{\lambda}, \overline{\lambda}]$ and $\gamma^{\dagger} = \infty$ on $(\overline{\lambda},\infty)$.
By construction $\gamma^\dagger:[0,\infty) \mapsto [0,\infty]$ is convex and we can define $D_{\gamma^\dagger}$.

Suppose that $C: I \mapsto \R_+$ is our increasing, lower semi-continuous cost function. Introduce $C^\dagger : \R_+ \mapsto [0,\infty]$ and $D_{C^\dagger}$ which we abbreviate to $D^\dagger$. Note that if $D^\dagger(z) < \ul$ then $z=0$, $D^\dagger(z)=0$ and $C^\dagger(0)=C^\dagger(\ul)=C(\ul)$. Summarising the important results we have:

%Write $D^\dagger$ as shorthand for $D_{\gamma^\dagger}$.

\begin{lem}
$\tilde{C} = \widetilde{C^{\dagger}}$. Moreover, for $z \in [0,\infty)$,
$C( (D^{\dagger}(z) \vee \underline{\lambda}) \wedge \overline{\lambda}) = C^\dagger(D^{\dagger}(z))$.
\label{lem:CCdagger}
\end{lem}

\begin{thm}
Suppose $I \subseteq [0,\infty)$ and let $C: I \mapsto \R$ be increasing, lower semi-continuous and such that $\lim_{\lambda \uparrow \infty} \frac{C(\lambda)}{\lambda} = \infty$. Let $G$ be an increasing, convex solution of \eqref{eq:Gdef} and suppose $G$ is of linear growth. Then $H=G$.
\label{thm:G=H2}
\end{thm}

\begin{proof}
Introduce $C^\dagger$, defined from $C$ as above, and let $H^\dagger$ be the solution of the unrestricted problem (ie $I^\dagger = [0,\infty)$) with (convex) cost function $C^\dagger$. Note that since $\tilde{C} = \tilde{C^\dagger}$ we have by Theorem~\ref{thm:H=G1} that $H^\dagger = G$. It remains to show that $H = H^\dagger$.

The inequality $H \leq H^\dagger$ is straight-forward: if $(\tau,\Lambda)$ is admissible for the interval $I$ and integrable for cost function $C$, then it is admissible for the interval $[0,\infty)$ and integrable for cost function $C^\dagger$; moreover $C \geq C^\dagger$, and so $H \leq H^\dagger$.

For the converse, let $\Lambda^\dagger = D^\dagger( (g(X_s) - G(X_s))_+)$ and $\tau^\dagger = T^{\Lambda^\dagger}_1$ be optimal for the problem with cost function $C^\dagger$. Note that $\Lambda^\dagger \leq \overline{\lambda}$ and that
\[ H^\dagger(x) = \E^x \left[ e^{- \beta \tau^\dagger} g(X_{\tau^\dagger}) - \int_0^{\tau^\dagger} e^{-\beta s} C^\dagger( \Lambda^\dagger_s) ds \right] \]

Define $\Lambda^* =\underline{\lambda} \vee \Lambda^\dagger$ and $\tau^* = \tau^\dagger$. Then, by Lemma~\ref{lem:CCdagger},
\[ C(\Lambda^*_s) = C( (D^\dagger( (g(X_s) - G(X_s))_+)\vee \underline{\lambda} )\wedge \overline{\lambda}) = C^\dagger( D^\dagger( (g(X_s) - G(X_s))_+)) = C^\dagger(\Lambda^\dagger_s). \]
Moreover, $\Lambda^* \in [\underline{\lambda},\overline{\lambda}]$ and is admissible for the original problem with admissibility interval $I$. Then
\[ H^\dagger(x) = \E^x \left[ e^{- \beta \tau^*} g(X_{\tau^*}) - \int_0^{\tau^*} e^{-\beta s} C( \Lambda^*_s) ds \right]  \leq H(x) . \]

\end{proof}

\begin{rem}
Note that $\Lambda^* \geq \Lambda^\dagger$ and we may have strict inequality if $\underline{\lambda}>0$. In that case, when $g(X_s) \leq G(X_s)$ we have $\Lambda^\dagger_s = 0$, but $\Lambda_s^* = \underline{\lambda}$. In particular, we may have $\tau^* > T^{\Lambda^*}_1$, and the agent does not sell at the first opportunity. See Section~\ref{ssec:subinterval}.
\end{rem}

\section{Concave cost functions}
In this section we provide a complementary result to Theorem~\ref{thm:H=G1} by considering a concave cost function $C$ (defined on $\sI = [0,\infty)$).

Suppose $C$ is increasing and concave on $[0,\infty)$. Then the greatest convex minorant $\breve{C}$ of $C$ is of the form
\[ \breve{C}(\lambda) = \delta + \epsilon \lambda   \]
for some constants $\delta,\epsilon \in [0,\infty)$. Then, $C$ and $\breve{C}$ have the same concave conjugates given by $\tilde{C}(z) := \inf_{\lambda > 0} \{ C(\lambda) - \lambda z \}$ where $\tilde{C}(z) = \delta$ for $z \leq \epsilon$ and $\tilde{C}(z)= - \infty$ for $z>\epsilon$.

From the heuristics section we expect the value function to solve \eqref{eq:Hheuristic}. Then we might expect that on $g-H < \epsilon$ we have
\begin{equation}
 \sL^X H - \beta H - \delta = 0.
\label{eq:Hleqe}
\end{equation}
On the other hand some care is needed to interpret $\sL^X H - \beta H  = \tilde{C}((g-H)_+)$ on the set $g-H>\epsilon$. In fact, as we argue in the following theorem, $H \geq g-\epsilon$ and on the set $H=g - \epsilon$ \eqref{eq:Hleqe} needs to be modified. We show that
$H = w_{K,\epsilon,\delta}$ where (recall~\eqref{eq:wKJD})
\begin{equation} w_{K, \epsilon, \delta}(x) = \sup_{\tau \in \sT([0,\infty))} \E^x \left[e^{-\beta \tau} \{(X_\tau - K)_+ - \epsilon \} - \delta \int_0^\tau e^{-\beta s} ds \right] . \label{eq:wKJD2} \end{equation}
The intuition is that when $H> g - \epsilon$ it is optimal to wait and to take $\Lambda=0$ at cost $\delta$ per unit time. However, on $H<g - \epsilon$ (and also when $H=g-\epsilon$) it is optimal to take $\Lambda$ as large as possible. Since there is no upper bound on $\Lambda$, this corresponds to taking $\Lambda$ infinite --- such a choice is inadmissible but can be approximated with ever larger finite values. Then, in the region where the agent wants to stop, if the stopping rate is large, say $N$, then the expected time to stop is $N^{-1}$, the cost incurred per unit time is $C(N) \approx \delta + \epsilon N$, and so the expected total cost of stopping is approximately $\frac{\delta + \epsilon N}{N} \approx \epsilon$. Effectively the agent can choose to sell (almost) instantaneously, for a fee or fixed transaction cost of $\epsilon$. This explains why the problem value is the same as the problem value for \eqref{eq:wKJD2}.

\begin{thm}
Let $I=[0,\infty)$ and let $C:I \mapsto \R_+$ be non-negative, increasing and concave. Suppose the greatest convex minorant $\breve{C}$ of $C(\lambda)$ is of the form $\breve{C}(\lambda) = \delta + \epsilon \lambda$ for non-negative constants $\delta$ and $\epsilon$.

%Suppose $w^{K,\epsilon,\delta}$ is the solution of the classical optimal stopping problem \eqref{eq:wKJD} for running cost $\delta$ and transaction cost $\epsilon$.

%Let $H$ be the value function for the problem with selling opportunities arising at event times of a Poisson process with cost function $C$.

Then $H(x) = w_{K,\epsilon,\delta}(x)$.
\label{thm:concave}
\end{thm}

\begin{proof}
First we show that for any integrable $\tau$ and $\Lambda$
\[ \E^x \left[ e^{-\beta \tau} (X_\tau - K)_+ - \int_0^\tau e^{-\beta s} C(\Lambda_s) ds \right] \leq w_{K,\epsilon,\delta}(x). \]
Then we show that there is a sequence of admissible strategies for which the value function converges to this upper bound.

We prove the result in the case $\epsilon \geq \delta/\beta$ when the cost of taking $\Lambda = 0$ is small relative to the proportional cost $C(\lambda)/\lambda$ associated with taking $\Lambda$ large. The proof in the case $\epsilon < \delta/\beta$ is similar, but slightly more complicated in certain verification steps, because the explicit form of $w^{K,\epsilon,\delta}$ is not so tractable.

When $\epsilon \geq \delta/\beta$ we have that $w = w_{K,\epsilon,\delta}$ is given by
\[ w(x) = \left\{ \begin{array}{ll} Ax^\theta - \frac{\delta}{\beta} & x \leq L  \\
                                    (x - K - \epsilon) & x > L \end{array} \right. , \]
where $L = \frac{\beta (K+\epsilon) - \delta}{\beta} \frac{\theta}{\theta - 1}$ and $A= \frac{1}{\theta} L^{1- \theta}$. Let $w^0(x) = w(x) \vee (x-K)_+$. Note that since $\frac{\beta}{\mu} > \theta$ we have $\frac{\theta}{\theta-1} > \frac{\beta}{\beta - \mu}$ and $L > \frac{\beta(K+\epsilon) - \delta}{\beta - \mu}$.

For fixed $\Lambda$ define $M^\Lambda = (M^\Lambda_t)_{t \geq 0}$ by $M^\Lambda_t = e^{- \int_0^t (\beta + \Lambda_s) ds} w(X_t) + \int_0^t  e^{- \int_0^s (\beta + \Lambda_u) du} [\Lambda_s w^0(X_s) - C(\Lambda_s)] ds$ and set $N_t = \int_0^t e^{- \int_0^s (\beta + \Lambda_u) du} \sigma X_s w'(X_s) dW_s$. Then $N=(N_t)_{t \geq 0}$ is a martingale and
\begin{equation}
dM^\Lambda_t  =  dN_t + e^{- \int_0^t (\beta + \Lambda_s) ds} \left[ \sL^X w - (\beta + \Lambda_t) w + \Lambda_t w^0 - C(\Lambda_t) \right] dt.
\label{eq:Mconcave}
\end{equation}
On $(0,L)$, $\sL^X w - \beta w = \delta$, and \eqref{eq:Mconcave} becomes
\[ dM^\Lambda_t  =  dN_t + e^{- \int_0^t (\beta + \Lambda_s) ds}[\delta - \Lambda_t w + \Lambda_t w^0 - C(\Lambda_t) ] dt \leq dN_t + e^{- \int_0^t (\beta + \Lambda_s) ds} [ \Lambda_t (w^0 - w - \epsilon)] dt \leq dN_t. \]
since $w^0 \leq w + \epsilon$.
Similarly, on $(L,\infty)$, $w(x) = (x - K) - \epsilon$ and since $L>K+\epsilon$, \eqref{eq:Mconcave} yields
\begin{eqnarray*}
 dM^\Lambda_t  & \leq & dN_t + e^{- \int_0^t (\beta + \Lambda_s) ds}[ \mu X_t - (\beta + \Lambda )(X_t - K - \epsilon) + \Lambda_t (X_t - K) - (\delta + \epsilon \Lambda_t )] dt \\
 & = & dN_t + e^{- \int_0^t (\beta + \Lambda_s) ds}[ (\mu - \beta) (X_t- L) + (\mu - \beta)L  + \beta(K+ \epsilon) - \delta ] dt \leq dN_t.
\end{eqnarray*}
%Since $ L > \frac{\beta(K+ \epsilon) - \delta}{\beta - \mu}$, again we have $dM^\Lambda_t \leq dN_t$.
Putting the two cases together we see that $M^\Lambda$ is a supermartingale for any strategy $\Lambda$.

The rest of the proof that $H \leq w$ follows exactly as in the the proofs of Lemma~\ref{lem:Gcomparison}, Lemma~\ref{lem:Ysupermg} and Proposition~\ref{prop:HleqG}, with $w$ replacing $G$. %In conclusion we have that $H \leq w$.

%{\bf Note to Matthew: What about the case $\epsilon < \delta/\beta$? Then the no-stopping region is $(\ell, L)$. We can get an expression for $L$ but we need to show $L > \frac{\beta(K+ \epsilon) - \delta}{\beta - \mu}$.}

Now we show that there is a sequence of strategies for which the value function converges to $w=w_{K,\epsilon,\delta}$. Since $\delta + \epsilon \lambda$ is the largest convex minorant of $C$ there exists $(\lambda_n)_{n \geq 1}$ with $\lambda_n \uparrow \infty$ such that $\frac{C(\lambda_n)}{\lambda_n} \rightarrow \epsilon$.

Consider first the strategy of a constant rate of search $\lambda_n$, with stopping at the first event time of the associated Poisson process. Let $\tilde{H}_n$ denote the associated value function. Then
\begin{eqnarray*}
\tilde{H}_n(x) & = & \E^x \left[ \int_0^\infty \lambda_n e^{-\lambda_n t} dt \left\{ e^{-\beta t} (X_t - K)_+ - \int_0^t e^{-\beta s} C(\lambda_n) ds \right\} \right] \\
& \geq & \int_0^\infty \lambda_n e^{-\lambda_n t} dt \left\{ e^{-\beta t} (xe^{\mu t} - K) - \int_0^t e^{-\beta s} C(\lambda_n) ds \right\} \\
& = & \int_0^\infty \lambda_n e^{-(\lambda_n + \beta) t} (xe^{\mu t} - K) dt -  \int_0^\infty e^{-\beta s} C(\lambda_n) ds \int_s^\infty \lambda_n e^{-\lambda_n t} dt \\
& = & \frac{\lambda_n}{\lambda_n + \beta - \mu} x - \frac{\lambda_n}{\lambda_n + \beta} K - \frac{1}{\lambda_n + \beta} C(\lambda_n)
\end{eqnarray*}
and $\tilde{H}_n(x) \rightarrow  x-K - \epsilon$ as $n \uparrow \infty$.
Suppose $\epsilon \geq  \delta/\beta$. Let $L = \frac{\beta (K+\epsilon) - \delta}{\beta} \frac{\theta}{\theta - 1}$ and let $\tau_L = \inf \{ u : X_u \geq L \}$. Consider the strategy with rate $\hat{\Lambda}_n = \lambda_n I_{ \{ t \geq \tau_L \} }$, for which selling occurs at the first event time of the Poisson process with this rate, and let $\hat{H}_n$ be the value function associated with this strategy.

For $x \geq L$ we have $\hat{H}_n(x) = \tilde{H}_n(x) \rightarrow x-K-\epsilon = w_{K,\epsilon,\delta}(x)$.

For $x < L$, we have $\E^x[ e^{- \beta \tau_L} ] = ( \frac{x}{L} )^\theta$ and
\begin{eqnarray*}
\hat{H}_n(x) & = & \E^x \left[ e^{- \beta \tau_L} \tilde{H}_n(L) - \int_0^{\tau_L} e^{-\beta s} C(0) ds \right] \\
& = & \E^x \left[ e^{- \beta \tau_L} \left( \tilde{H}_n(L) + \frac{C(0)}{\beta} \right) - \frac{C(0)}{\beta}  \right] \\
& = & \left( \frac{x}{L} \right)^\theta \left[ \tilde{H}_n (L) + \frac{\delta}{\beta} \right] - \frac{\delta}{\beta} \\
& \rightarrow & w_{K, \epsilon, \beta}(x),
\end{eqnarray*}
where the last line follows from the definition of $L$ and some algebra.
\end{proof}

\subsection{An example}
In this example we consider a cost function of the form $C(\lambda) = \sqrt{\lambda}$. Then a (plausibly) good strategy is to take $\Lambda_t = 0$ if $X_t < L^* = \frac{\theta}{\theta - 1}$ and $\Lambda_t$ very large otherwise. It is immediate that the value function $H$ satisfies $H \leq w$; conversely, it is clear from Figure~\ref{fig:concave} that there exist strategies for which the value function is arbitrarily close to $w$.
  \begin{figure}[H] \center
\includegraphics [width=0.5\linewidth]  {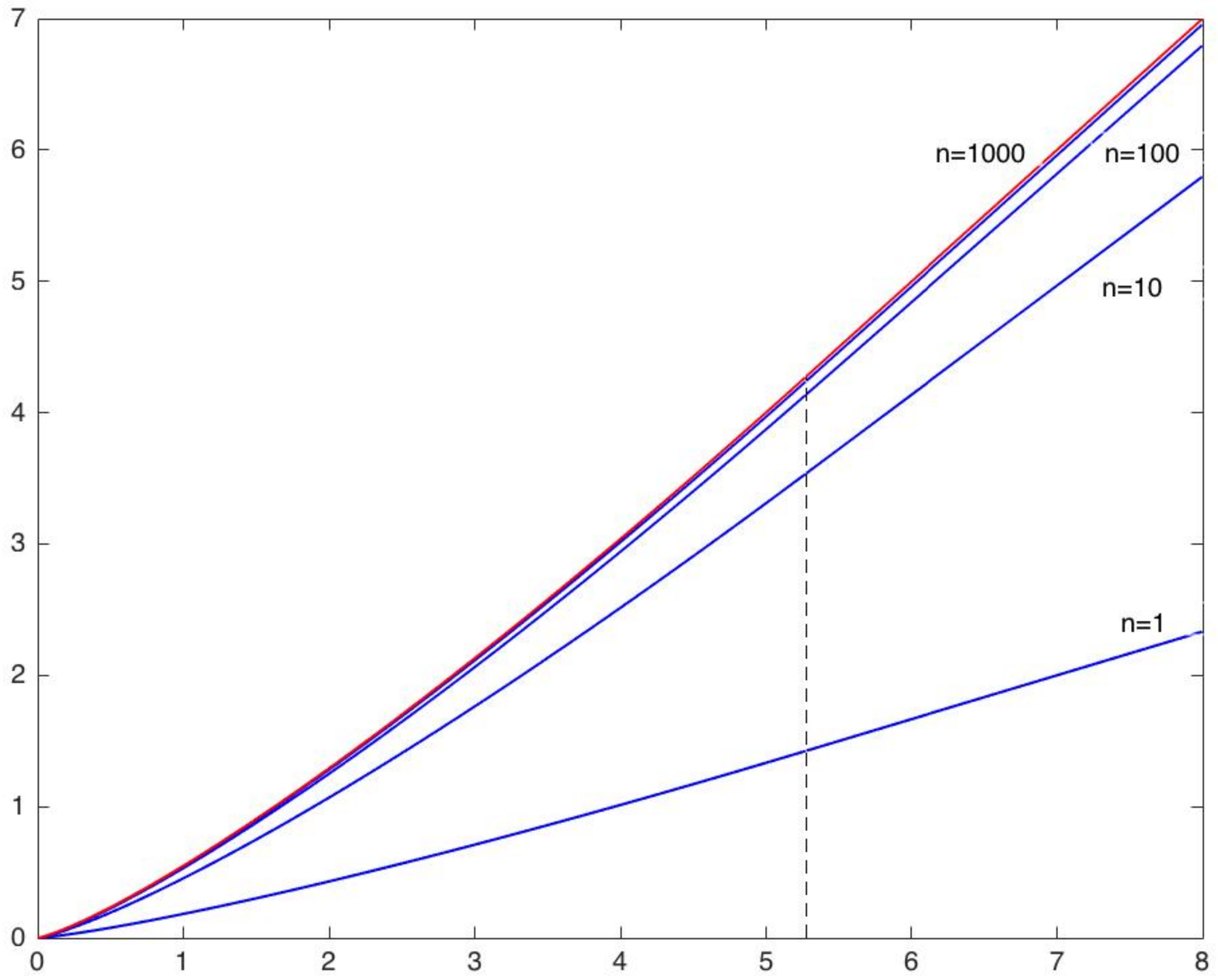}
\caption{ $(\beta,\mu,\sigma,K) =(5,3,3,1)$; the highest line is $w = w_{K,0,0}$, and the other lines are the value functions under the rate function $\Lambda_n(x) = n I_{ \{x \geq L^* \} }$.}
\label{fig:concave}
\end{figure}

\section{Further Examples}
\label{sec:examples}

\subsection{Addition of a linear cost}
\label{ssec:lc}
%When $a=0$ we have that $G \geq 0$. It follows that $(((x-K)_+ - G)_+ - b)_+ = ((x-(K+b))_+ - G)_+$ and hence the problem with strike $K$ and cost $C(\lambda) = b \lambda + \frac{c}{2} \lambda^2$ has the same value as the problem with strike $K$ and cost $C(\lambda) = \frac{c}{2} \lambda^2$. The same is true for $a>0$ provided that $G > -b$. Hence, for many parameter combinations

Let $C_0$ be a convex, lower semi-continuous, increasing cost function, and consider the impact of adding a linear cost to $C_0$; in particular, let $C_b : \R_+ \mapsto \R_+$ be given by $C_b(\lambda) = C_0(\lambda) + \lambda b$ for $b>0$.

Then the concave conjugates are such that $\tilde{C}_b(z ) = \tilde{C}_0((z-b)_+)$.

Suppose further that $G$, the solution of \eqref{eq:Gdef} of linear growth, is such that $G \geq 0$ on $\R_+$. The problem solution in the case of a purely quadratic cost function (recall Section~\ref{sssec:pure}) has this property. Then
\[ ( \{ (x -K)_+ - G \}_+ - b)_+ = \{ (x - (K+b))_+ - G \}_+ . \]
It follows that
\[ \tilde{C}_b( \{ (x-K)_+ - G \}_+ ) = \tilde{C}_0(( \{ (x -K)_+ - G \}_+ - b)_+) = \tilde{C}_0 (\{ (x - (K+b))_+ - G \}_+ ) \]
and then that the value function for a payoff $(x-K)^+$ with cost function $C_b$ is identical to the value function for a cost function $C_0(x)$ but with modified payoff $(x - (K+b))_+$.

Note that we see a similar result in the expansion \eqref{eq:Gexpansion} for $G$ in the large $x$ regime.

\subsection{Quadratic costs with positive fixed cost}
In this section we seek to generalise the results of Section~\ref{sssec:pure} on purely quadratic cost functions to other quadratic cost functions. In view of the results in Section~\ref{ssec:lc} the focus is on adding a positive intercept term, rather than a linear cost. Indeed the focus is on cost functions of the form $C(\lambda) = a + \frac{c}{2} \lambda^2$ for $a>0$.

In this section we will take $a$ and $c$ fixed and compare the cost fucntions $C_0(\lambda) = \frac{c}{2} \lambda^2$, $C_1(\lambda) = a + \frac{c}{2} \lambda^2$ and $C_>(\lambda) = a I_{ \{ \lambda > 0 \} } + \frac{c}{2} \lambda^2$. The difference between the last two cases is that in the final case, not searching at all incurs zero cost, whereas in the middle case, there is a fixed cost which applies irrespective of whether there is a positive rate of searching for offers or not.

In Section~\ref{sssec:pure} we saw that $H_0$, the value function for the cost $C_0(\lambda) = \frac{c}{2} \lambda^2$, solves
\[ \sL^X H_0 - \beta H_0 = \frac{[(g-H_0)_+]^2}{2c}. \]
There is a threshold $L$ with $L>K$, such that $H_0 > g$ on $(0,L)$ and $H_0<g$ on $(L,\infty)$. On $(0,L)$ we have that $H_0(x) = (L-K)\frac{x^\theta}{L^\theta}$; on $[L,\infty)$, $H_0$ solves $\frac{1}{2} \sigma^2 x^2 h'' + \mu x h' - \beta h = \frac{1}{2c} (x-K - h)^2$ subject to initial conditions $H_0(L) = (L-K)$ and $H_0'(L) = \theta\frac{L-K}{L}$. We adjust $L$ until we find a solution for which $H_0$ is of linear growth at infinity.

Now consider $C_1$ with asscoiated value function $H_1$. When $X$ is very small, there is little prospect of $X$ ever rising above $K$. Nonetheless the agent faces a fixed cost, even if she does not search for offers. It will be cheaper to search for offers, because although the payoff is zero when a candidate purchaser is found, it is then possible in our model to stop paying the fixed cost.

%{\bf Note to both of us. Does this mean there are possibly different solutions to ODE with different behaviour at 0. We choose the best?}

Suppose $X=0$. If the agent chooses to search for buyers at rate $\lambda$ then the expected time until a buyer is found is $\lambda^{-1}$. The expected discounted cost until a buyer is found is
\[ \int_0^\infty \lambda e^{-\lambda s} \int_0^s e^{-\beta u} \left( a + \frac{c}{2} \lambda^2 \right) du =\frac{a + \frac{c}{2} \lambda^2}{\beta + \lambda}. \]
This is minimised by the choice $\lambda= \lambda_*$ where $\lambda_* = \sqrt{{\beta^2} + \frac{2a}{c}} - \beta$ and the minimal cost is $h^-_*$ where
\[ h^-_* = \frac{a + \frac{c}{2} \lambda_*^2}{\beta + \lambda_*} = c \lambda_* = c \left[ \sqrt{{\beta^2} + \frac{2a}{c}} - \beta  \right] . \]
Then $H_1(0)= - h^-_*$.
(Another way to see this is to note that at 0 we expect $\sL^X H_1=0$ and therefore $H_1(0)$ to solve $- \beta h = \tilde{C}(- h) = a - \frac{h^2}{2c}$.)
Then, the value function $H_1$ is such that there exists $\ell$ and $L$ with $0<\ell < K < L < \infty$ such that $H_1$ is $C^1$ with $H_1<0$ on $(0,\ell)$, $H_1(x)>(x-K)_+$ on $(\ell,L)$ and $H_1(x)<(x-K)_+$ on $(L,\infty)$ and such that $H_1$ satisfies
\[ \sL^X h - \beta h = \left\{ \begin{array}{lll} a - \frac{1}{2c} h^2    & &  x < \ell; \\
                                             a                       & \; &\ell < x < L; \\
                                             a - \frac{1}{2c} (g-h)^2                & & L < x.
                                             \end{array} \right. \]
See Figure~\ref{fig:C1}. Considering $H_1$ on $(\ell,L)$ we have $H_1(x) = A x^\theta + B x^\phi - \frac{a}{\beta}$ for some constants $A$ and $B$ chosen so that $H_1(\ell)=0$ and $H_1(L) = (L-K)$:
\[ A = \frac{L^{-\phi}(L-K  + \frac{a}{\beta}) - \ell^{-\phi} \frac{a}{\beta}}{L^{\theta-\phi} - \ell^{\theta-\phi}}, \hspace{20mm} B =  \frac{ \ell^{-\phi} L^{\theta - \phi} \frac{a}{\beta} - \ell^{\theta-\phi} L^{-\phi}(L-K + \frac{a}{\beta})}{L^{\theta-\phi} - \ell^{\theta-\phi}} . \]
Then for general $\ell$ and $L$ we can use value matching and smooth fit at $\ell$ and $L$ to construct a solution on $(0,\infty)$. Finally, we adjust $\ell$ and $L$ until $H_1(0)=-h^-_*$ and $H_1$ has linear growth.

%  \begin{figure}[H] \center
%\includegraphics [width=0.5\linewidth]  {concave.pdf}
%\caption{{\bf Wrong picture. Matthew to draw correct pictures - one of the value function and one of the associated best $\lambda$} $(\beta,\mu,\sigma,K,a,c) =(5,3,3,1,1,2)$; plot of the value function for the cost $C_1(\lambda) = 1 + \lambda^2$ and the associated best $\lambda$. When $X$ is small, optimal to search for offers.}
%\label{fig:C1}
%\end{figure}

\begin{figure*}[!ht]
    \centering
    \begin{subfigure}[t]{0.45\textwidth}
        \includegraphics[width=1\textwidth]{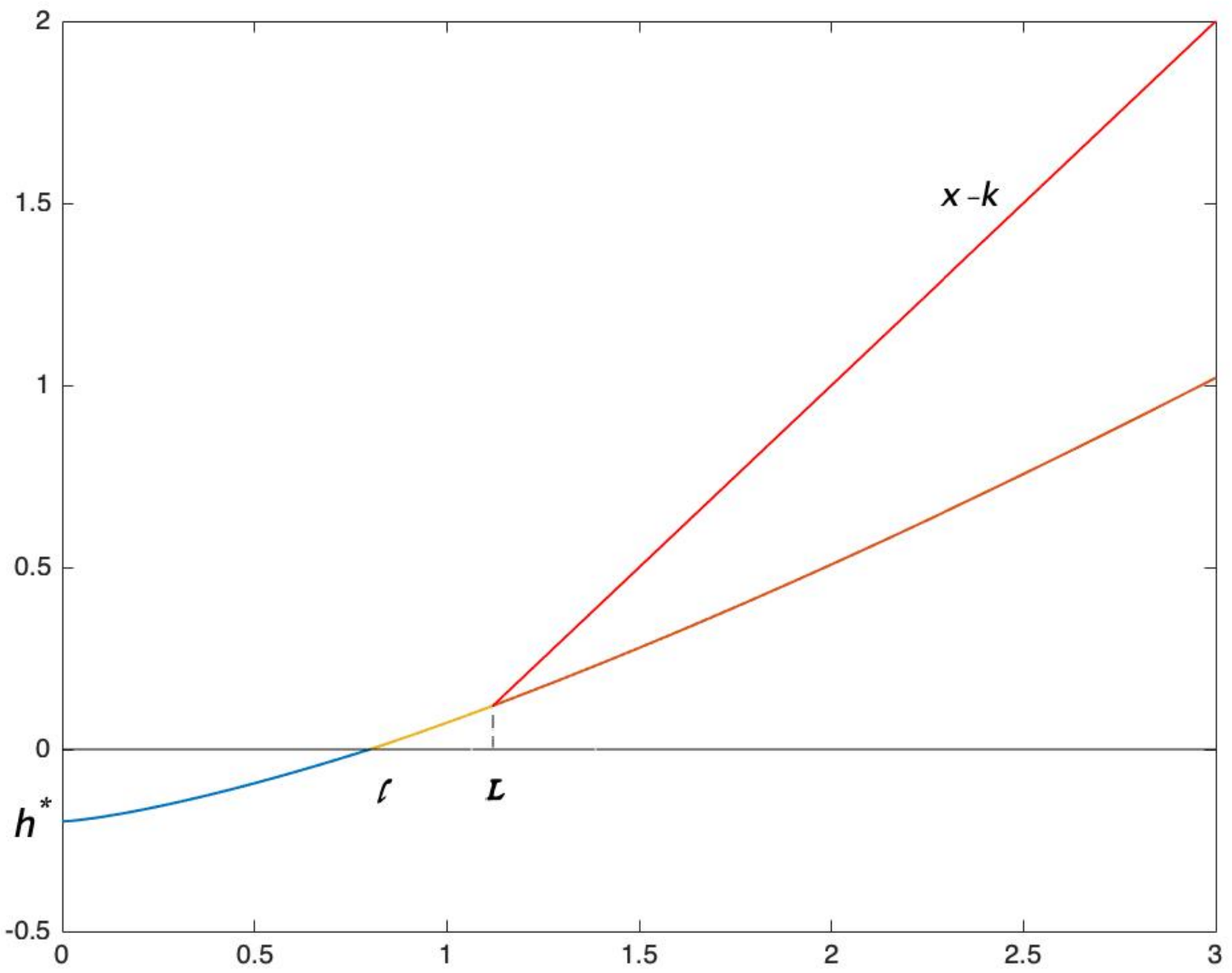}
        \caption{The value function $H_1(x)$.}
    \end{subfigure}%
    ~
    \begin{subfigure}[t]{0.45\textwidth}
        \centering
        \includegraphics[width=1\textwidth]{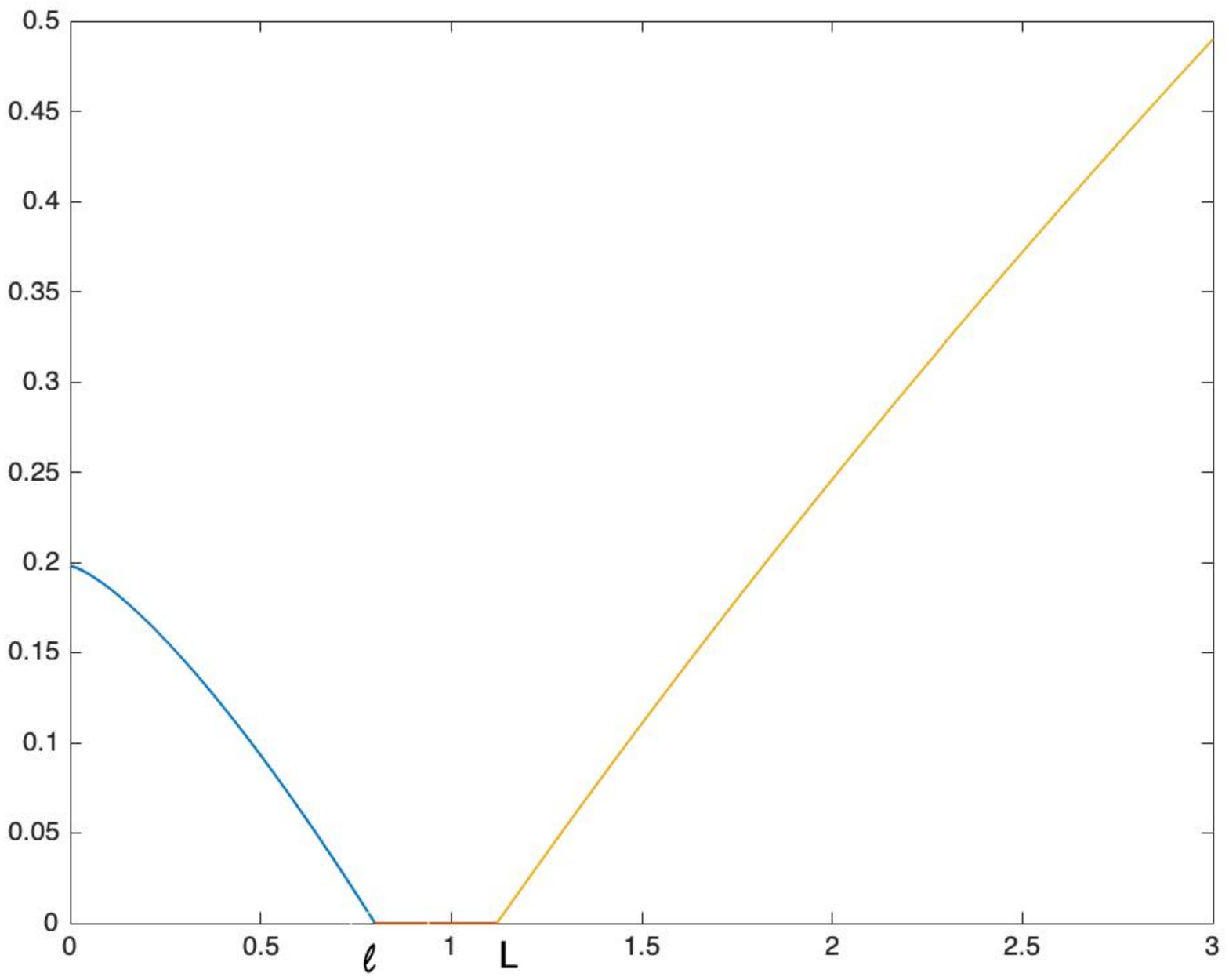}
        \caption{The optimal rate $\Lambda_1^*(x)$.}
        %\label{fig:zoomin}

    \end{subfigure}
    \caption{$(\beta,\mu,\sigma,K) =(5, 3 ,2,1)$. The cost function is $C_1(\lambda)=1+\lambda^2$. The left figure shows the value function, and the right figure the optimal stopping rate. There are two critical thresholds $\ell= \ell^*$ and $L=L^*$.
    }
    \label{fig:C1}
 \end{figure*}
In Figure~\ref{fig:C1} we plot the value function and optimal rate for the Poisson process for $C_1(\lambda) = 1 + \lambda^2$. There are two critical thresholds $\ell^*$ and $L^*$ with $0 < \ell^* < K < L^*$.
Above $L^*$ the agent would like to stop in order to receive the payoff $(x-K)$, and is willing to expend effort to try to generate selling opportunities in order to receive the payoff before discounting reduces the worth. Below $\ell^*$ the agent would like to stop, even though the payoff is zero, and is willing to expend effort to generate stopping opportunities in order to limit the costs they incur prior to stopping.
Between $\ell^*$ and $L^*$ the agent does not expend any effort searching for offers and would not accept any offers which were received.

Now consider the cost function $C_>(\lambda)=aI_{\{ \lambda>0 \} } + \frac{c}{2} \lambda^2$ with associated value function $H_>$. We have $\widetilde{C}_>(z) = 0$ for $z \leq \sqrt{2ac}$ and $\widetilde{C}_> = a - \frac{z^2}{2c}$ for $z \geq \sqrt{2ac}$.
As in the pure quadratic case, there is always the option of taking $\Lambda \equiv 0$ at zero cost, so that the value function is non-negative. It follows that $H_>(0)=0$.
There is a threshold $L$ below which the agent does not search for offers. But, this threshold is not the boundary between the sets $\{x : H_>(x)>g(x) \}$ and $\{ x : H_>(x) < g(x) \}$, since when $g(x)-H_>(x)$ is small, it is still preferable to take $\Lambda = 0$, rather than to incur the cost of strictly positive $\lambda$. Instead $L$ separates the sets $\{x : H_>(x)>g(x) - \sqrt{2ac} \}$ and $\{ x : H_>(x) < g(x) - \sqrt{2ac} \}$.
We find that there is a threshold $L$ with $L > K$ such that on $(0,L)$, $H_>$ solves $\sL^X h - \beta h = a$. At $L$ we have $H_>(L) = (L-K-\sqrt{2ac})$ and it follows that on $(0,L)$ we have $H_>(x) = \frac{L-K-\sqrt{2ac}}{L^\theta} x^\theta$.
Then, on $(L,\infty)$, $H_>$ solves $\sL^X h - \beta h = a - \frac{(x-K-h)^2}{2c}$, subject to value matching and smooth fit conditions at $x=L$. Finally, we adjust the value of the threshold $L$ until $H$ is of linear growth for large $x$.

%  \begin{figure}[H] \center
%\includegraphics [width=0.5\linewidth]  {concave.pdf}
%\caption{{\bf Wrong picture. Matthew to draw correct pictures - one of the value function and one of the associated best $\lambda$} $(\beta,\mu,\sigma,K,a,c) =(5,3,3,1,1,2)$; plot of the value function for the cost $C_1(\lambda) = I_{ \{ \lambda > 0 \} } + \lambda^2$ and the associated best $\lambda$. Threshold is {\bf not} where $V=L$.}
%\label{fig:C>}
%\end{figure}

\begin{figure*}[!ht]
    \centering
    \begin{subfigure}[t]{0.45\textwidth}
        \includegraphics[width=1\textwidth]{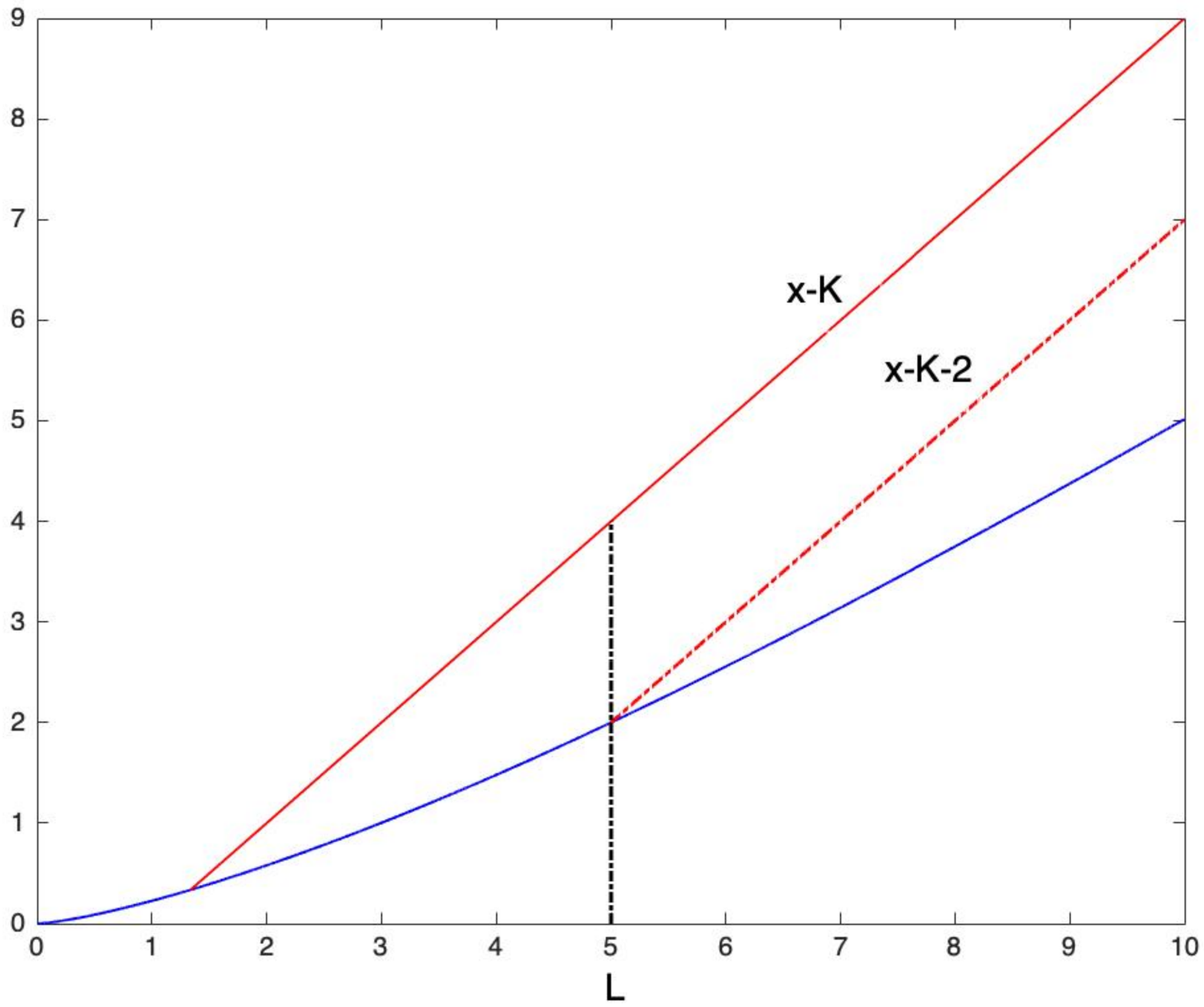}
        \caption{The value function $H_>(x)$}
    \end{subfigure}%
    ~
    \begin{subfigure}[t]{0.45\textwidth}
        \centering
        \includegraphics[width=1\textwidth]{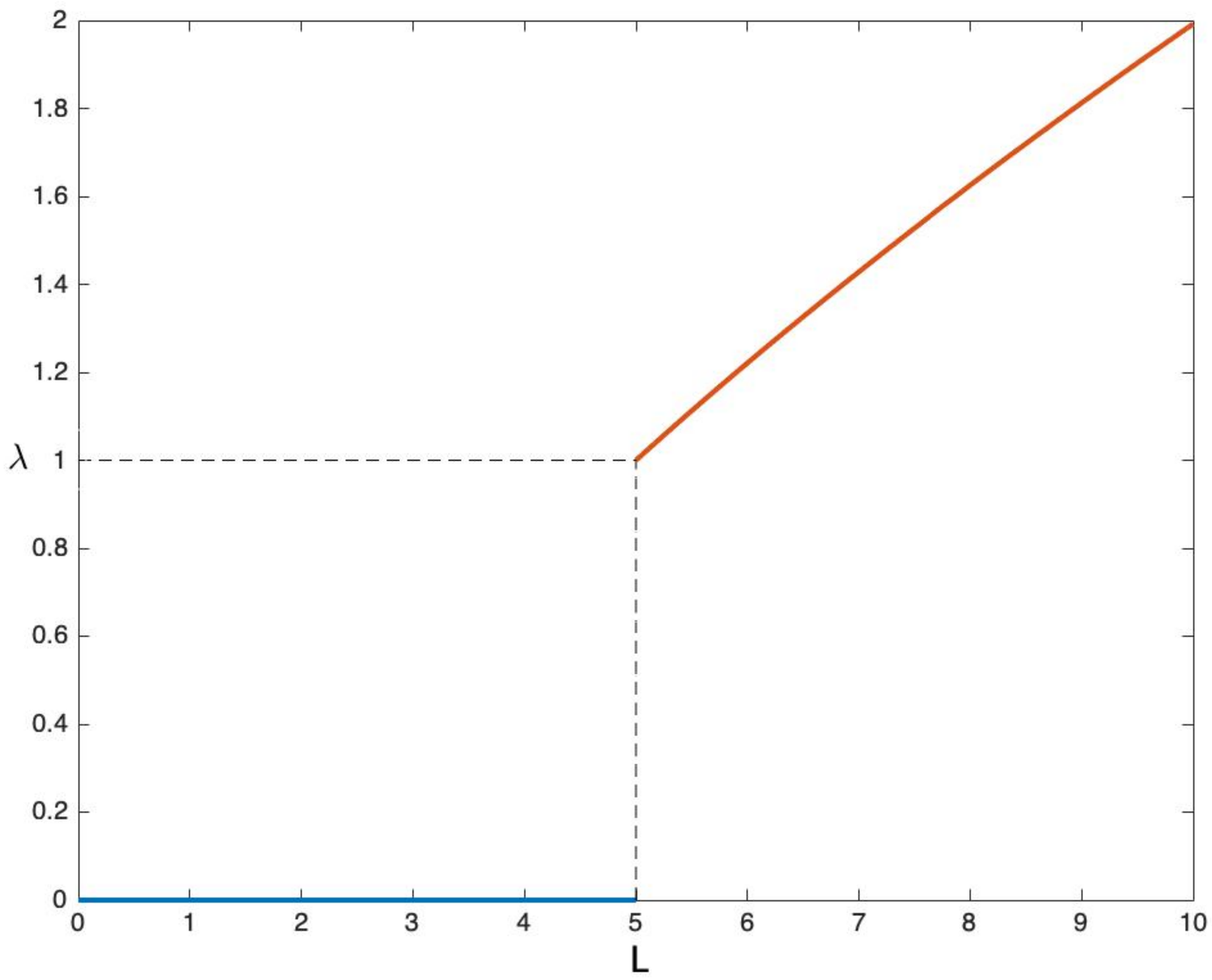}
        \caption{The optimal rate $\Lambda^*_>(x)$}
        %\label{fig:zoomin}

    \end{subfigure}
    \caption{$(\beta,\mu,\sigma,K) =(5, 3 ,2,1)$. The cost function is $C_>(\lambda) = I_{\{ \lambda>0 \} } + \lambda^2$. The highest convex minorant is $\breve{C}_>(\lambda) =  \lambda + [(\lambda - 1)_+]^2$. (Here we use the fact that $\sqrt{{2a}{c}}=2$.)
     }
    \label{fig:C>}
 \end{figure*}

In Figure~\ref{fig:C>} we plot the value function $H_>$ and optimal rate $\Lambda^*_>$. We see that $\Lambda^*_>$ never takes values in $(0,1)$ where $C_> > \breve{C}_>$. Either it is optimal to spend a non-negligible amount of effort on searching for candidate buyers, or it is optimal to spend no effort.

%  \begin{figure}[H] \center
%\includegraphics [width=0.5\linewidth]  {concave.pdf}
%\caption{{\bf Wrong picture. Matthew to draw } $(\beta,\mu,\sigma,K,a,c) =(5,3,3,1,1,2)$; plot of the value functions $H_0, H_1, H_>$ and best strategies $\Lambda_0, \Lambda_1, \Lambda_>$.}
%\label{fig:Ccomparison}
%\end{figure}

\begin{figure*}[!ht]
    \centering
    \begin{subfigure}[t]{0.45\textwidth}
        \includegraphics[width=1\textwidth]{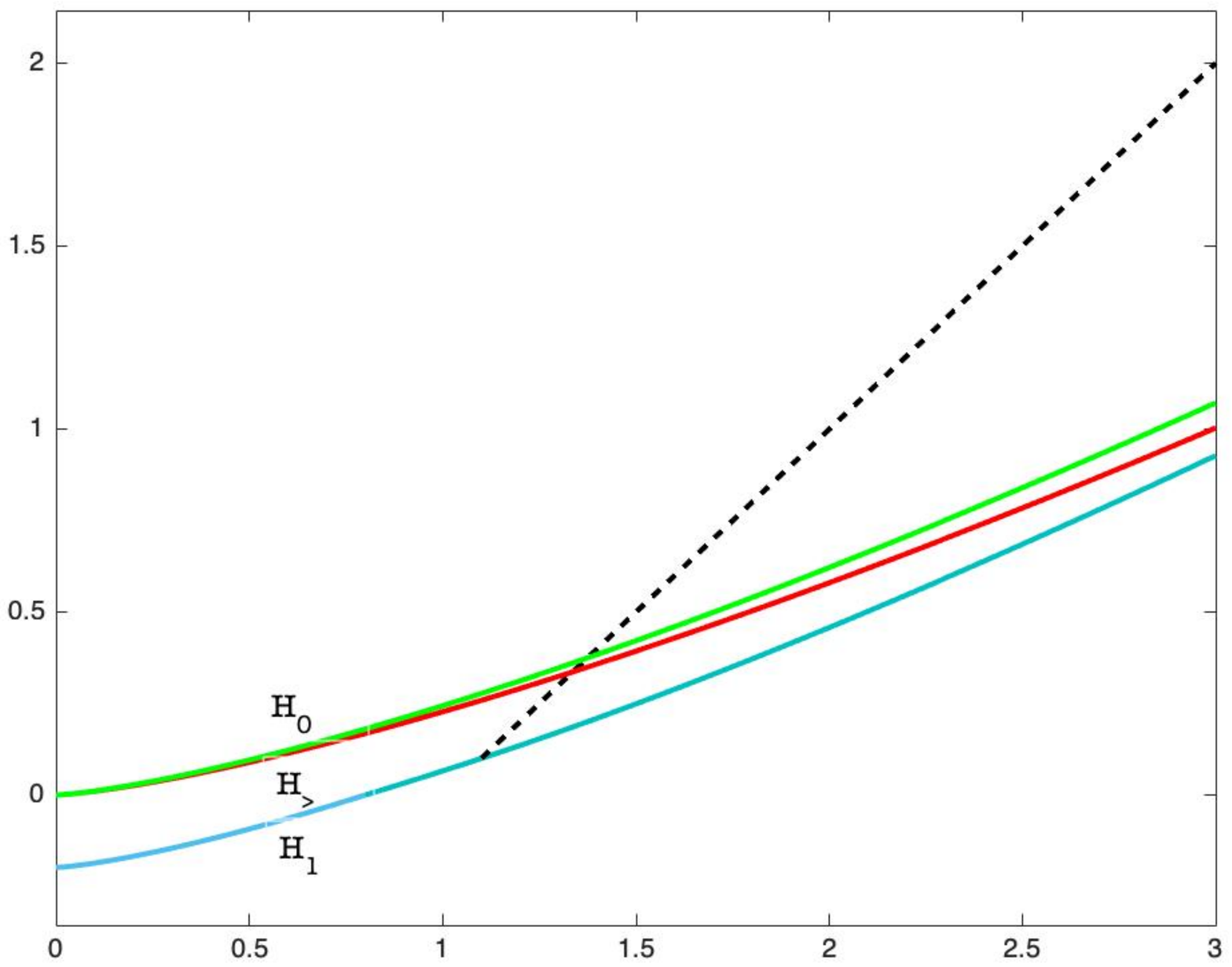}
        \caption{Comparison of the value functions}
    \end{subfigure}%
    ~
    \begin{subfigure}[t]{0.45\textwidth}
        \centering
        \includegraphics[width=1\textwidth]{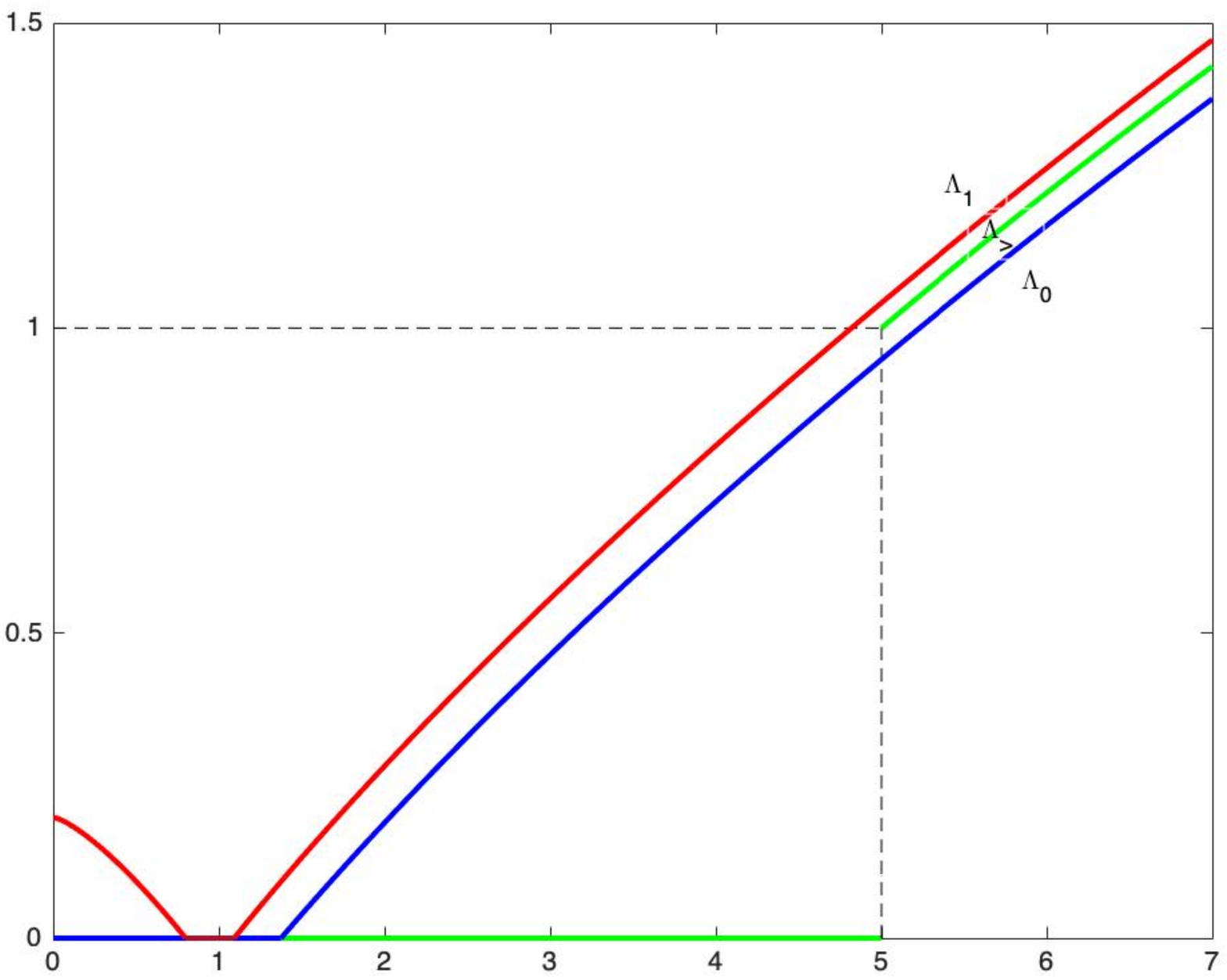}
        \caption{Comparison of the optimal stopping rates}
        %\label{fig:zoomin}

    \end{subfigure}
    \caption{$(\beta,\mu,\sigma) =(5, 3 ,2,1)$. The cost functions we consider are $C_0(\lambda) = \lambda^2$, $C_>(\lambda) = I_{\{ \lambda>0 \}} + \lambda^2$ and $C_1(\lambda) = 1 + \lambda^2$. The left figure plots the value functions under optimal behaviour, and the right figure plots the optimal rates for the Poisson process. For $x>5$ we have $\Lambda_1^* > \Lambda_>^* > \Lambda_0^*$. For small $x$, $\Lambda_1^* > 0 = \Lambda_>^* = \Lambda_0$.}

    \label{fig:comparison}
 \end{figure*}

Figure~\ref{fig:comparison} compares the value functions and optimal rates for the Poisson process for the three cost functions $C_0(\lambda) = \lambda^2$, $C_>(\lambda) = I_{\{ \lambda>0 \}} + \lambda^2$ and $C_1(\lambda) = 1 + \lambda^2$. Since $C_0 \leq C_> \leq C_1$ we must have that $H_0 \geq H_> \geq H_1$ and we see that away from $x=0$ this inequality is strict.
Indeed, especially for small $x$, $H_0$ and $H_>$ are close in value. The differences in optimal strategies are more marked. For large $x$ the fact that $H_0>H_> >H_1$ means that $\Lambda^*_0 < \Lambda^*_> < \Lambda^*_1$, and thus that even though $C_1>C_0$, the agent searches at a higher rate under $C_1$ than under $C_0$. Note that, we only have $\Lambda^*_> >0$ for $x$ above a critical value (in our case, approximately 5). Conversely, for $C_1$ there is a second region where $\Lambda_1>0$, namely where $x$ is small.

\subsection{Cost functions defined on a subset of $\R_+$}
\label{ssec:subinterval}
In this section we consider the case where there is a strictly positive lower bound on the rate at which offers are received. In fact, in our example the optimal rate of offers takes values in a two-point set. Nonetheless, we see a rich range of behaviours.

Suppose $\Lambda$ takes values in $[\ul,\ol]$ where $0<\ul<\ol<\infty$ and suppose $C : [\ul,\ol] \mapsto \R_+$ is increasing and concave.
Introduce $\breve{C}:[\ul,\ol] \mapsto [0,\infty)$ defined by $\breve{C}(\lambda) = C(\ul) + \frac{\lambda - \ul}{\ol-\ul} (C(\ol)-C(\ul))$. Finally introduce $C^\dagger: [0,\infty) \mapsto [0,\infty]$ by
\[ C^\dagger (\lambda) = \left\{ \begin{array}{lcl} C(\ul) & \; & \lambda < \ul, \\
                                                   \breve{C}(\lambda) && \ul \leq \lambda \leq \ol,  \\
                                                   \infty    && \ol < \lambda . \end{array} \right. \]
Write $a = C(\ul)$ and $b = \frac{ (C(\ol)-C(\ul))}{\ol-\ul}$. Then $C^\dagger$ has concave conjugate $\tilde{C}^\dagger (z) = a - \ul z$ for $z \leq b$ and $\tilde{C}^\dagger(z) = a - b \ul - (z-b) \ol$ for $z>b$.

Suppose first that $C(\ul)=a=0$. Then the value function $H$ is positive, increasing and $C^1$ and satisfies
\[ \sL^X h - \beta h = \left\{ \begin{array}{lcl} 0 & \; & x < L, \\
                                                 - \ul (g - h) && L \leq x \leq M,   \\
 - b \ul - \ol(g - h - b) && M < x, \end{array} \right. \]
where $L$ and $M$ are constants satisfying $0<K<L<M$ which must be found as part of the solution, and are such that $h(x) > (x-K)$ on $(0,L)$, $(x-K) > h(x) > x-K-b$ on $(L,M)$ and $(x-K-b)>h(x)$ on $(M,\infty)$. See Figure~\ref{fig:subinterval1}.

Fix $L$ and consider constructing a solution to the above problem with $H(0)$ bounded. On $(0,L)$ we have that $H(x) = Ax^\theta +B x^\phi$ and the requirement that $H$ is bounded means that $B=0$ and then $A= (L- K)L^{-\theta}$. We then use the $C^1$ continuity of $H$ at $L$
to find the constants $C$ and $D$ in the expression for $H$ over $(L,M)$:
\begin{equation}
\label{eq:HLM} H(x) = C x^{\underline{\theta}} + Dx^{\underline{\phi}} + \frac{\ul}{\ul + \beta - \mu} x - \frac{K \ul}{\ul +\beta}.
\end{equation}
Here $\underline{\phi}$, $\underline{\theta}$  with $\underline{\phi} < 0 <1 < \underline{\theta}$ are solutions to $Q_{\ul}(\cdot)=0$ where
$Q_\lambda(\psi) = \frac{1}{2} \sigma^2 \psi(\psi - 1) + \mu \psi - (\beta + \lambda)$.
In turn, we can find the value of $M=M(L)$ where $H$ given by \eqref{eq:HLM} crosses the line $y(x)=x-K-b$, and then value matching at $M$ gives us the value of $E$ in the expression for $H$ over $[M,\infty)$:
\[ H(x) = E x^{\overline{\phi}} + \frac{\ol}{\ol + \beta - \mu} x - \frac{(K+b) \ol- b \ul}{\ol +\beta} \]
where $\overline{\phi}$ is the negative root of $Q_{\ol}(\cdot)=0$. (There is no term of the form $x^{\overline{\theta}}$ since $H$ must be of linear growth at infinity.) Finally, we can solve for $L$ by matching derivatives of $H$ at $M$.

%  \begin{figure}[H] \center
%\includegraphics [width=0.5\linewidth]  {concave.pdf}
%\caption{{\bf Wrong picture. Matthew to draw correct pictures - one of the value function and one of the associated best $\lambda$} $(\beta,\mu,\sigma,K,a,c) =(5,3,3,1,1,2)$; plot of the value function for the cost $C_1(\lambda) = I_{ \{ \lambda > 0 \} } + \lambda^2$ and the associated best $\lambda$. Threshold is {\bf not} where $V=L$.}
%\label{fig:subinterval}
%\end{figure}

\begin{figure*}[!ht]
    \centering
    \begin{subfigure}[t]{0.45\textwidth}
        \includegraphics[width=1\textwidth]{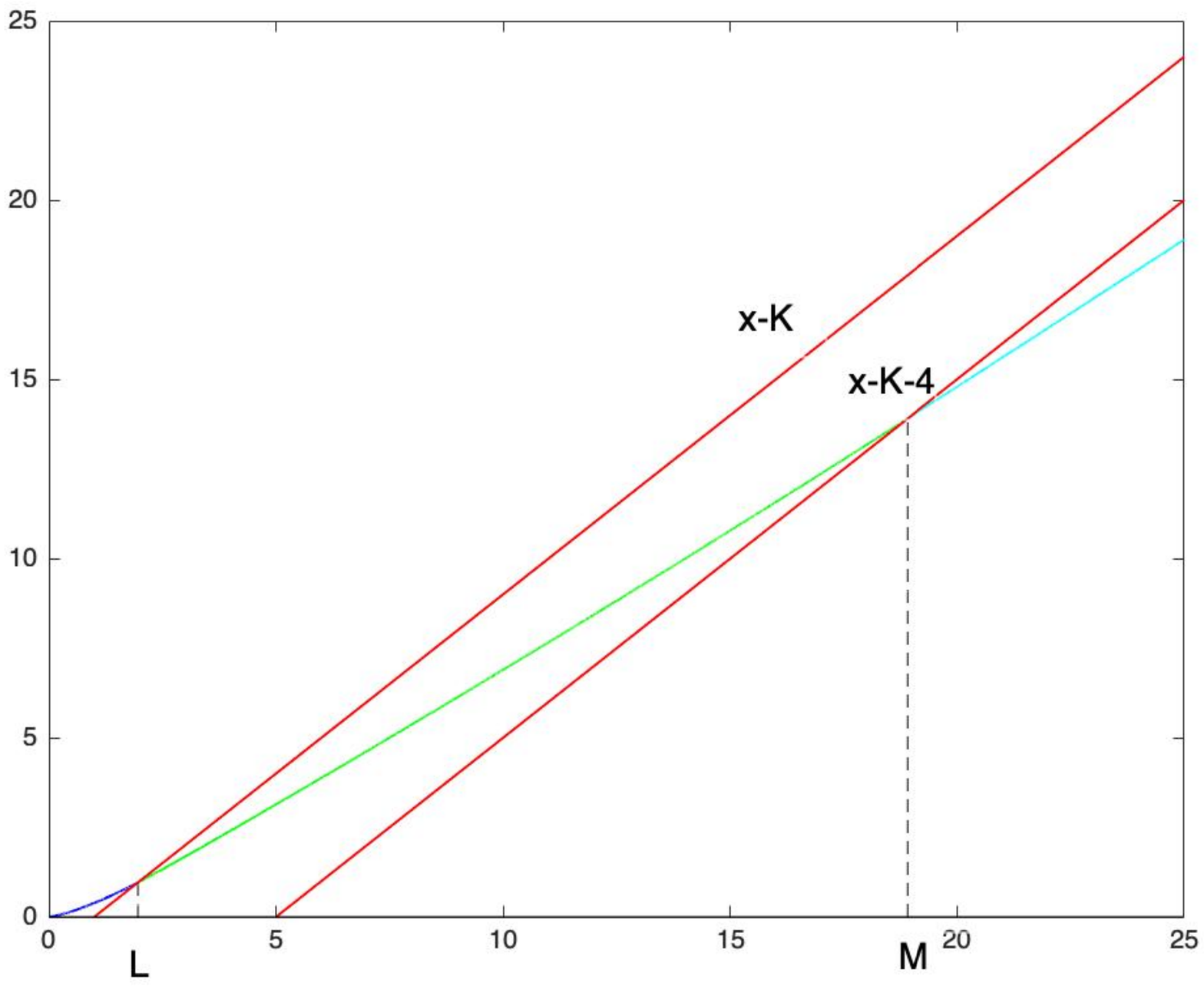}
        \caption{The value function $H$.}
    \end{subfigure}%
    ~
    \begin{subfigure}[t]{0.45\textwidth}
        \centering
        \includegraphics[width=1\textwidth]{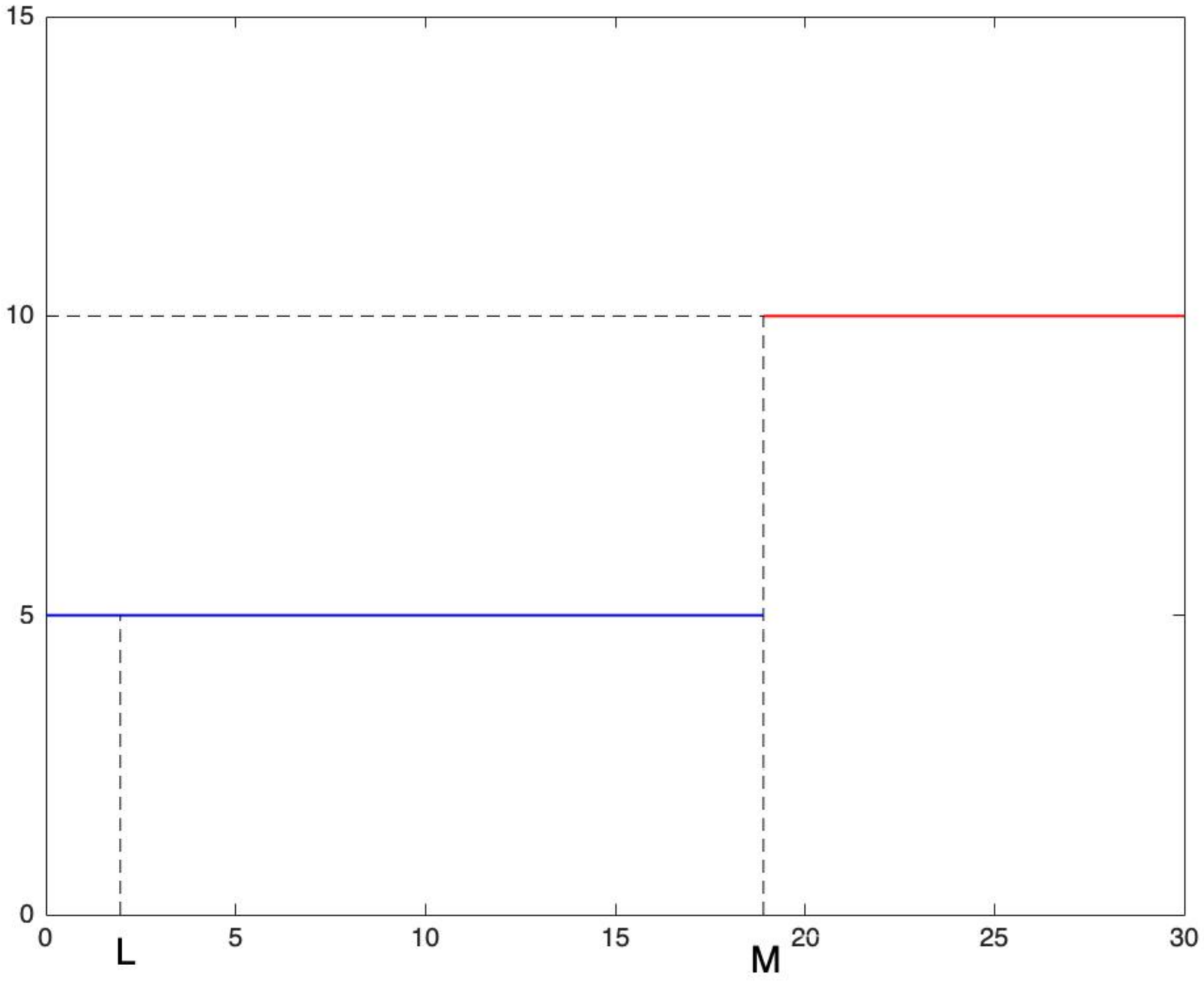}
        \caption{The optimal rate $\Lambda^*$.}
        %\label{fig:zoomin}

    \end{subfigure}
    \caption{$(\beta,\mu,\sigma,K,\ul, \ol, C(\ul), C(\ol)) =(5, 3 ,2,1,5,10, 0, 20)$. Note that $b = \frac{ (C(\ol)-C(\ul))}{\ol-\ul}=4$.
    The left figure plots the value function and the right figure plots the optimal rate function.
    $\Lambda$ is constrained to lie in $[5,10]$, and the cost function is $20 I_{ \{\lambda>5 \}}$. We see that $\Lambda^*$ takes values in $\{5,10\}$.}
    \label{fig:subinterval1}
 \end{figure*}

Figure~\ref{fig:subinterval1} plots the value function and the optimal rate function. The state space splits into three regions. On $x>M$ the asset is considerably in-the-money and the agent is prepared to pay the cost to generate a higher rate of selling opportunities. When $x$ is not quite so large, and $L<x<M$, the agent is not prepared to pay this extra cost, but will sell if opportunities arise. However, if $x<L$ then selling opportunities will arise (we must have $\Lambda \geq \ul$) but the agent will forgo them. Ideally the agent would choose $\Lambda=0$, but this is not possible. Instead the agent takes $\Lambda = \ul$, but synthesises a rate of zero, by rejecting all offers.

When $C(\ul)>0$, the agent will not pay the fixed cost indefinitely when $X$ is small. The behaviour for large $X$ is unchanged, but the agent will now stop if offers arrive when the value of continuing is negative, including when $X$ is near zero. There are two cases depending on whether $\frac{C(\ul)}{\ul+\beta} \leq \frac{C(\ol)}{\ol+\beta}$ or otherwise. In the former case, when $X$ is small it is cheaper to pay the lower cost and to stop if opportunities arise, than to pay the higher cost with the hope of stopping sooner. In the latter case, the comparison is reversed. We find that $H$ solves
\[ \sL^X h - \beta h  = \tilde{C}((g-h)_+) \]
subject to $h(0) = - \min_{\lambda \in \{ \ul, \ol \} } \{ \frac{C(\lambda)}{\lambda + \beta} \}$ and the fact that $h$ is of linear growth at infinity. The solution is smooth, except at points where $\tilde{C}((g-h)_+)$ is not differentiable. This may be at $K$ where $g$ is not differentiable, or when $g=h$, or, since $\tilde{C}$ is non-differentiable at $b$, when $g-h=b$.

Figure~\ref{fig:subinterval2} shows the value function and the optimal search rate in the case where $\frac{C(\ul)}{\ul+\beta} \leq \frac{C(\ol)}{\ol+\beta}$. This means that when $x$ is small the agent expends as little effort as possible searching for offers, although they do accept any offers which arrive. There is also a critical threshold $M$, beyond which it is optimal to put maximum effort into searching for offers. There are then two sub-cases depending on whether costs are small or large. If costs are large then the agent will always accept any offer which comes along (Figure~\ref{fig:subinterval2}(c) and (d)). However, when costs are small (Figure~\ref{fig:subinterval2}(a) and (b)), there is a region $(\ell, L)$ over which $h(x)>g(x)=(x-K)_+$. Then, as in the region $(0,L)$ when $C(\ul)=0$, even when there is an offer the agent chooses to reject it. Effectively, the agent creates a zero rate of offers by thinning out all the events of the Poisson process.

\begin{figure*}[!ht]
    \centering

    \begin{subfigure}[t]{0.45\textwidth}
        \includegraphics[width=1\textwidth]{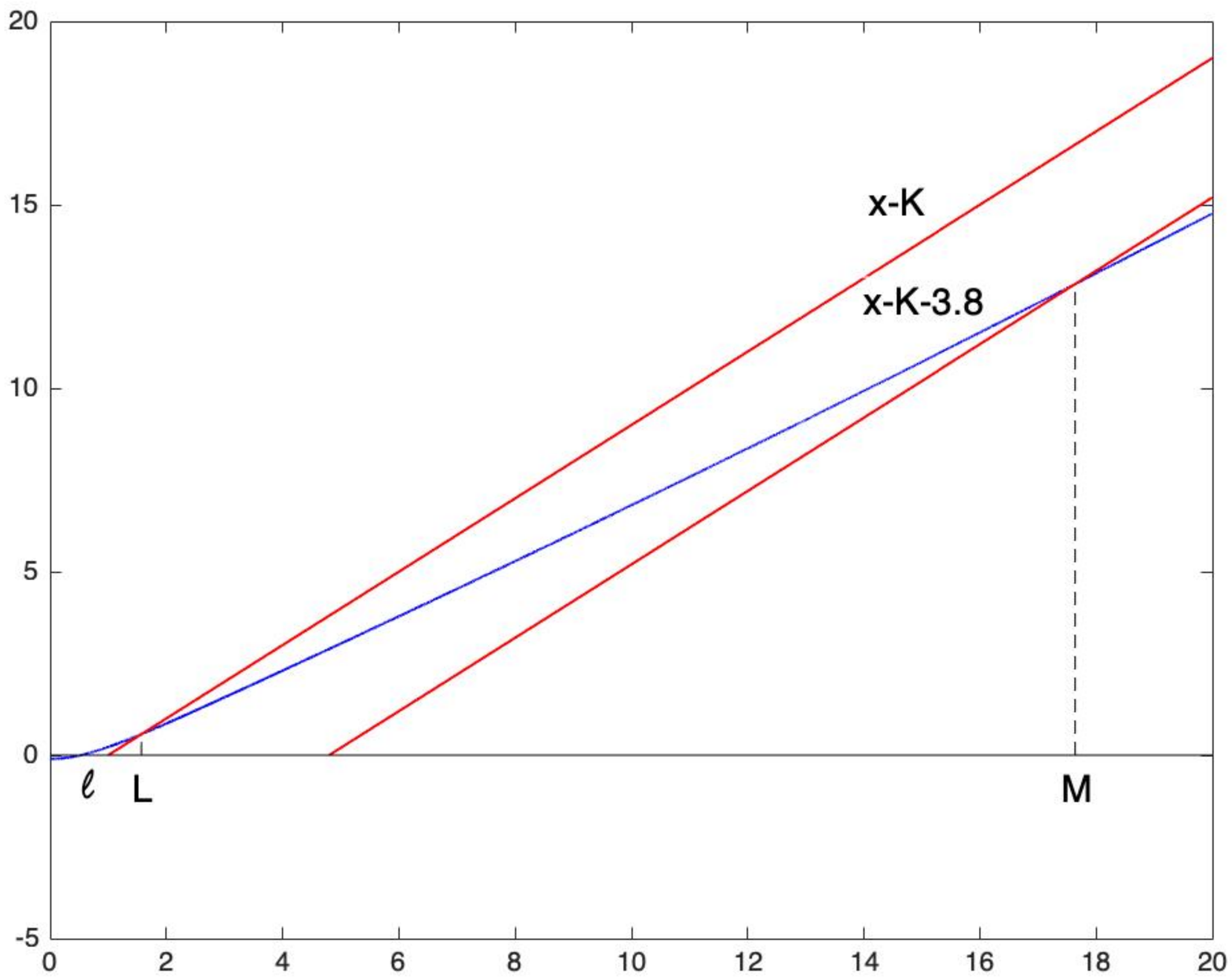}
        \caption{The value function $H$ in the case $(C(\ul)=1,C(\ol)=20)$.}
    \end{subfigure}%
    ~
    \begin{subfigure}[t]{0.45\textwidth}
        \centering
        \includegraphics[width=1\textwidth]{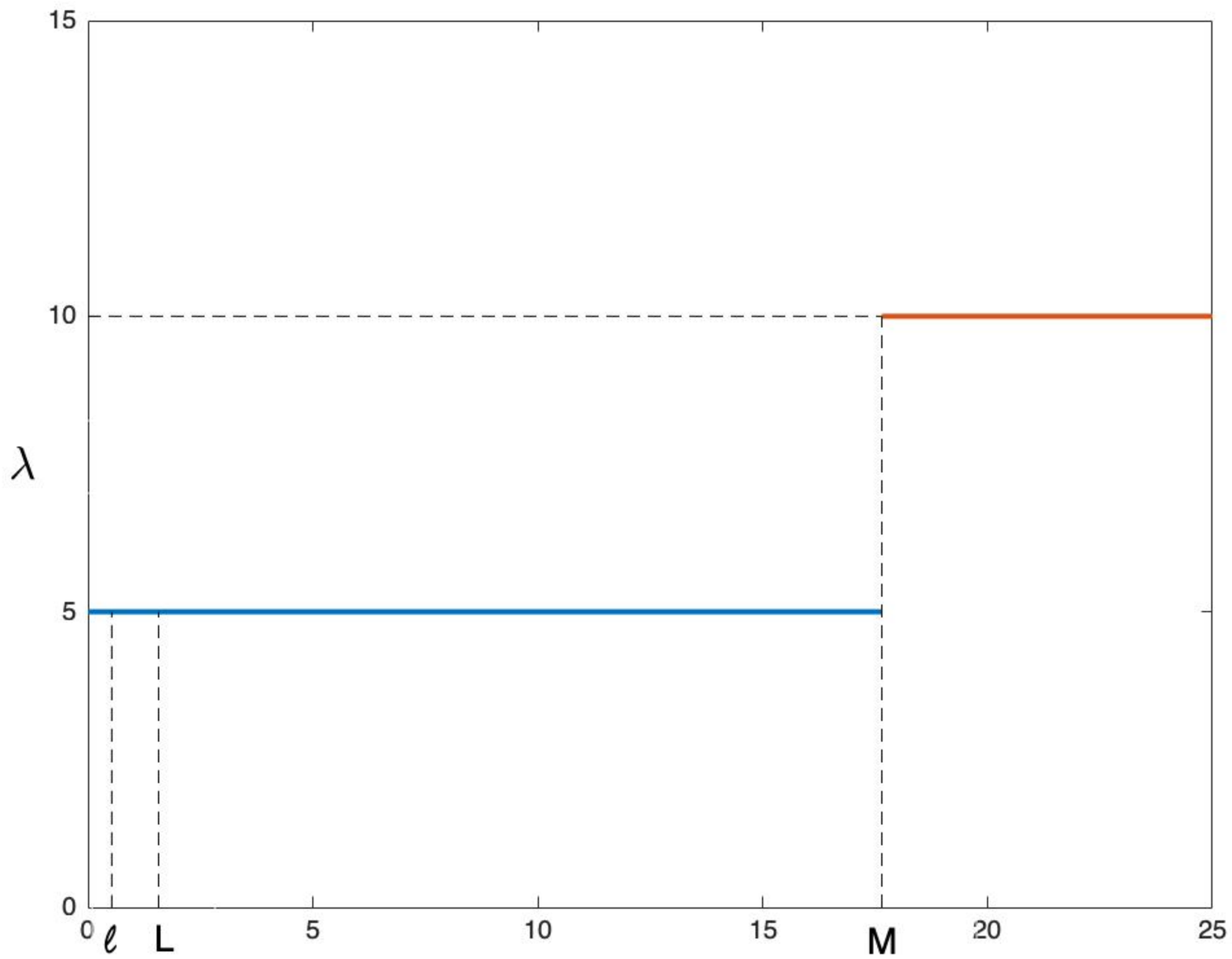}
        \caption{The optimal rate $\Lambda^*$ in the case $(C(\ul)=1,C(\ol)=20)$.}

    \end{subfigure}
    \begin{subfigure}[t]{0.45\textwidth}
        \includegraphics[width=1\textwidth]{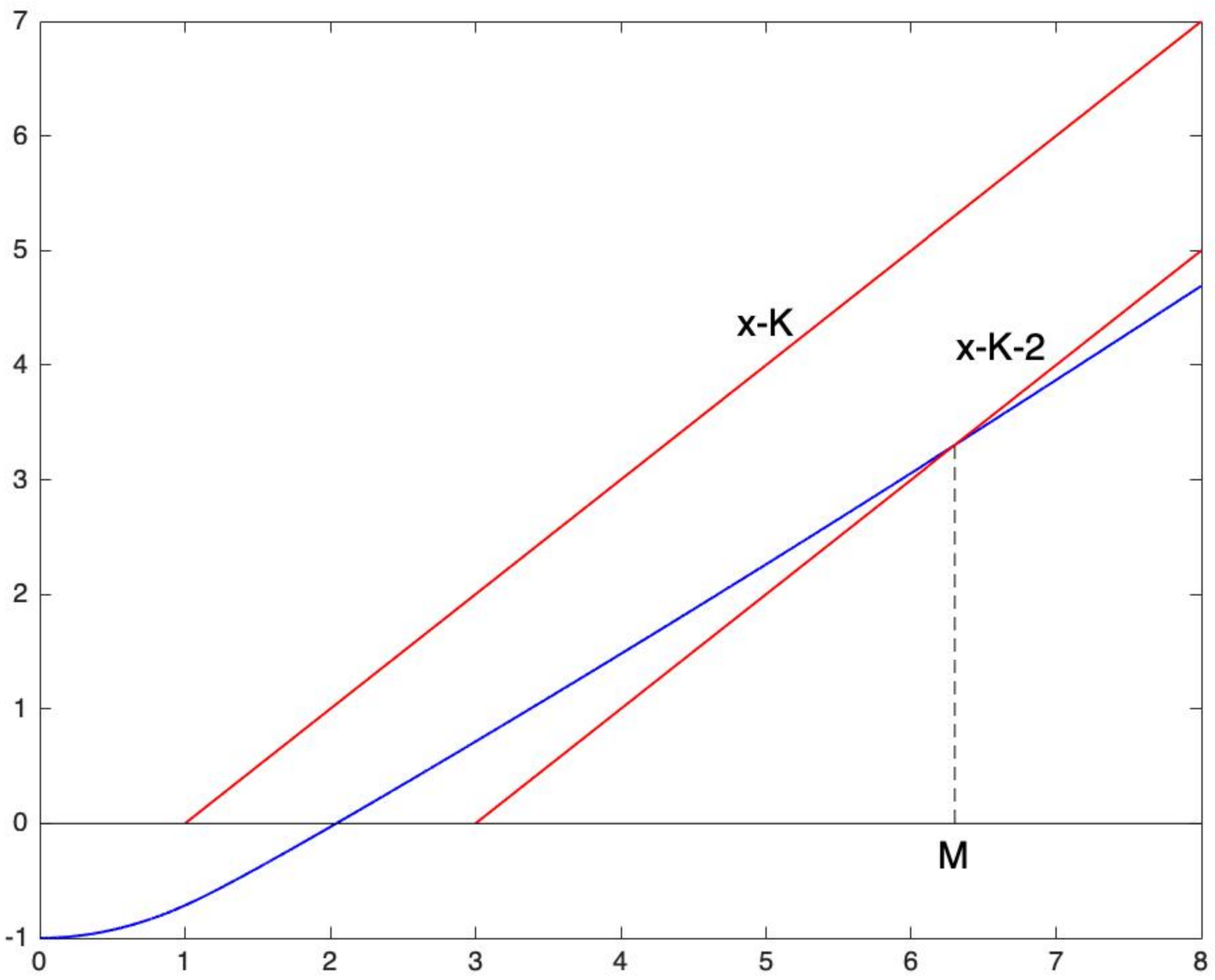}
        \caption{The value function $H$ in the case $(C(\ul)=10,C(\ol)=20)$.}
    \end{subfigure}%
    ~
    \begin{subfigure}[t]{0.45\textwidth}
        \centering
        \includegraphics[width=1\textwidth]{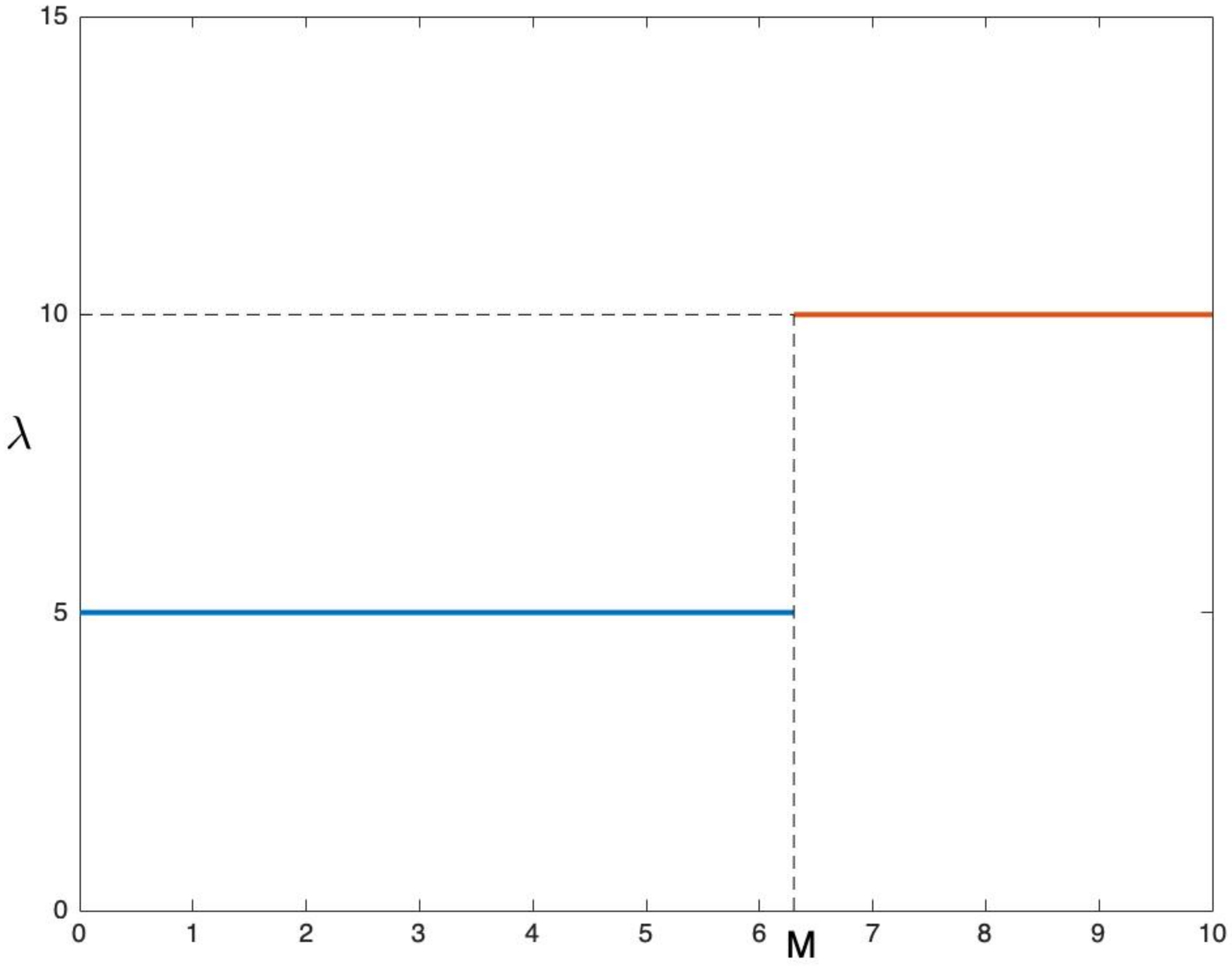}
        \caption{The optimal rate $\Lambda^*$ in the case $(C(\ul)=10,C(\ol)=20)$.}

    \end{subfigure}
    \caption{$(\beta,\mu,\sigma,K,\ul, \ol) =(5, 3 ,2,1,5,10)$.
    The left panels plot the value function and the right panels plot the optimal rate function. In each row $\frac{C(\ul)}{\ul + \beta} < \frac{C(\ul)}{\ul + \beta}$.
    In the case of lower costs ($C(\ul)=1$) there is a region $(\ell,L)$ where $H(x)>g(x)$ and the agent chooses to continue rather than to stop.}

    \label{fig:subinterval2}
 \end{figure*}

Figure~\ref{fig:subinterval3} shows the value function and the optimal search rate in the case where $\frac{C(\ul)}{\ul+\beta} > \frac{C(\ol)}{\ol+\beta}$.
Then, necessarily, $b = \frac{C(\ol)-C(\ul)}{\ol - \ul} < \frac{C(\ol)}{\ol + \beta}$. When $x$ is small the agent searches at the maximum rate to generate an offer as quickly as possible. Necessarily $H(0) < - b$. If costs are large enough, then $H(x) < (x-k)_+ - b$ for all $x$, see Figure~\ref{fig:subinterval3}(a) and (b). Then the agent wants to stop as soon as possible, and is prepared to pay the higher cost rate in order to facilitate this. As costs decrease, we may have $(x-k)_+ -b \leq H(x)$ for some $x$, whilst the inequality $H(x) < (x-k)_+$ remains true, see Figure~\ref{fig:subinterval3}(c) and (d). Then there is a region $(m,M)$ over which the optimal strategy is $\Lambda^*(x)=\ul$. The agent still accepts any offer which is made. Finally, if costs are small enough we find that there is a neighbourhood $(\ell, L)$ of $K$ for which $H(x)>(x-K)_+$. Then, on $(\ell, L)$ the agent takes $\Lambda^*(x) = \ul$, but chooses to continue rather than stop if any offers are made.

\begin{figure*}[!ht]
    \centering
    \begin{subfigure}[t]{0.45\textwidth}
        \includegraphics[width=1\textwidth]{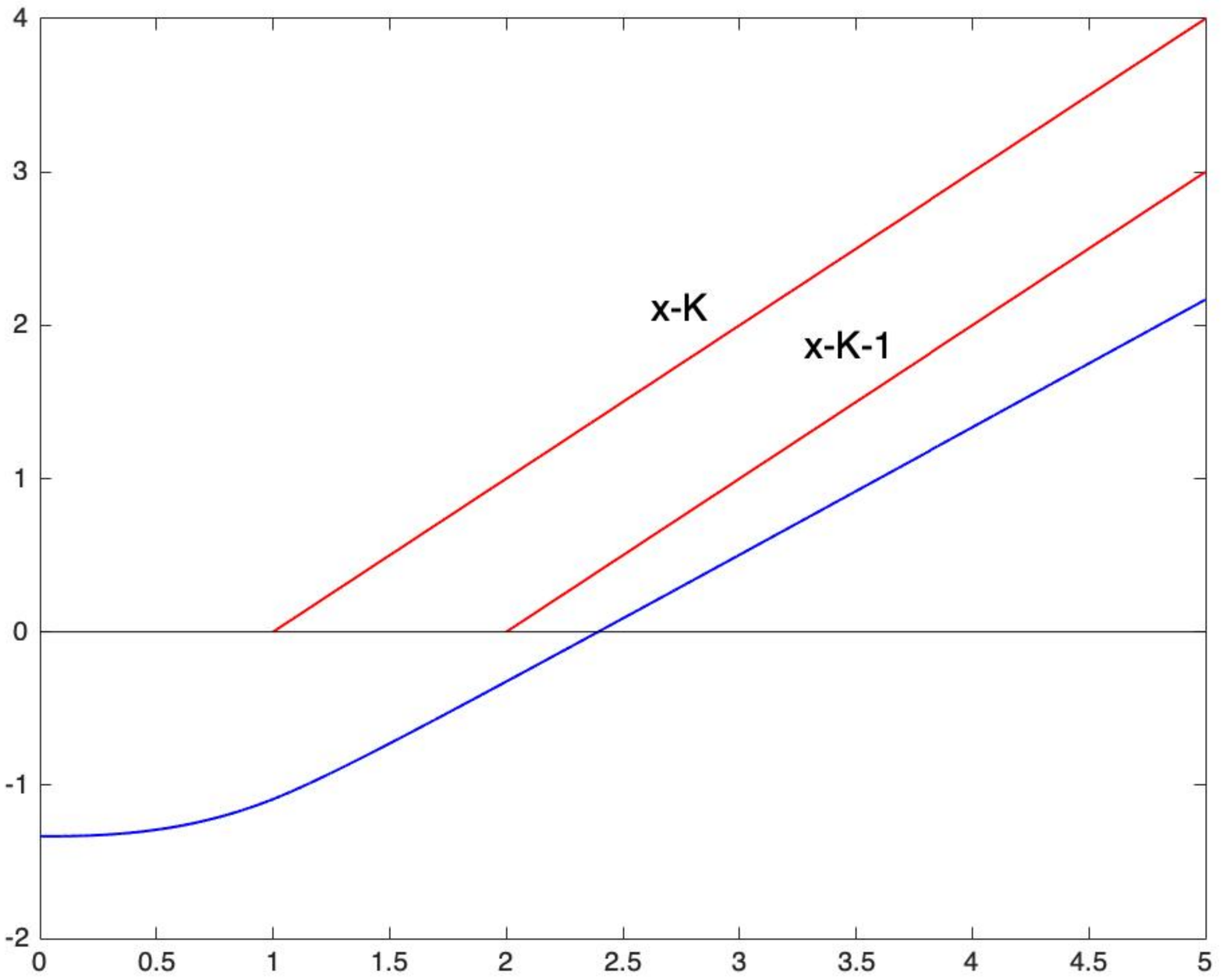}
        \caption{The value function $H$ in the case $(C(\ul)=15,C(\ol)=20)$. $b=1$.}
    \end{subfigure}%
    ~
    \begin{subfigure}[t]{0.45\textwidth}
        \centering
        \includegraphics[width=1\textwidth]{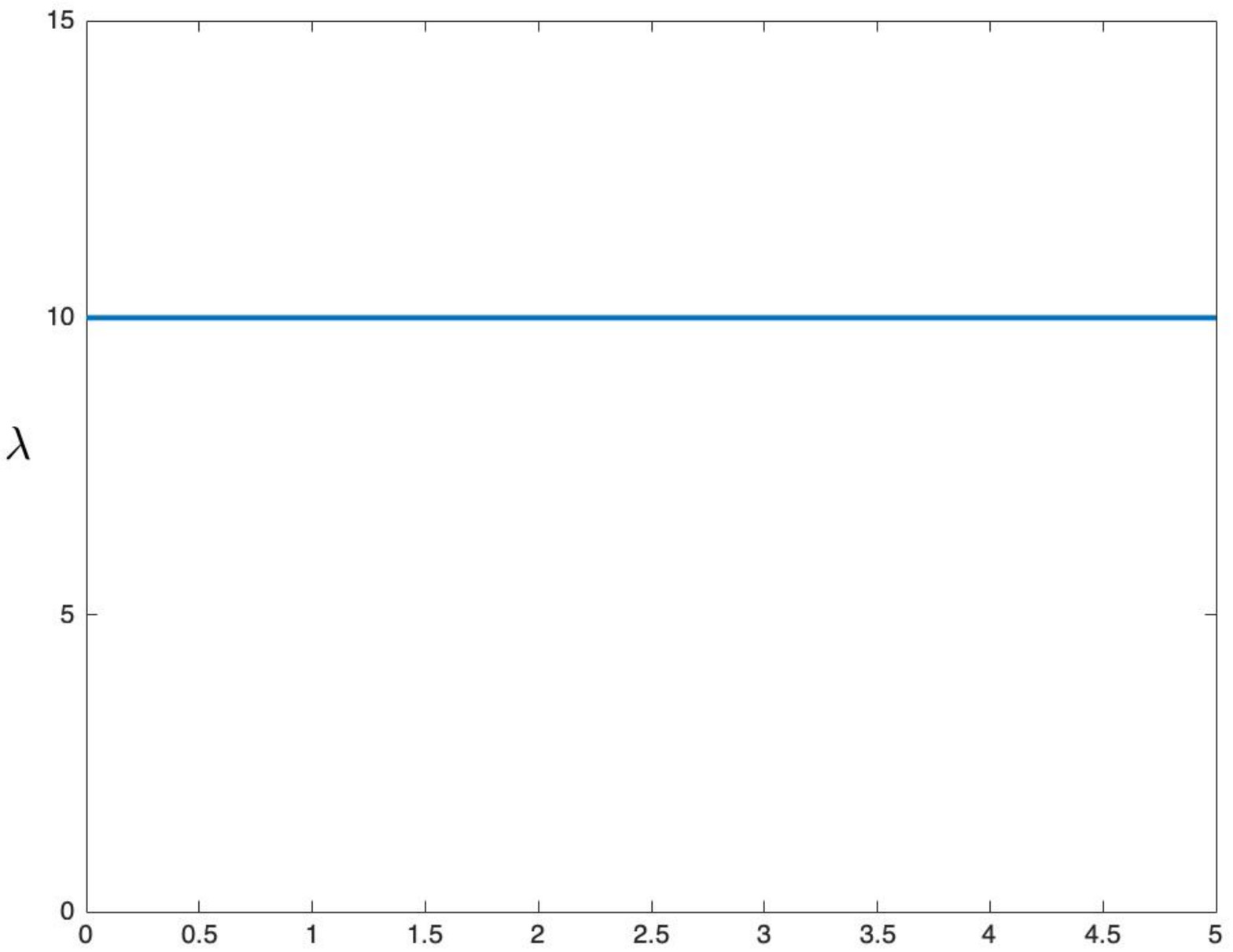}
        \caption{The optimal rate $\Lambda^*$ in the case $(C(\ul)=15,C(\ol)=20)$.}

    \end{subfigure}

    \begin{subfigure}[t]{0.45\textwidth}
        \includegraphics[width=1\textwidth]{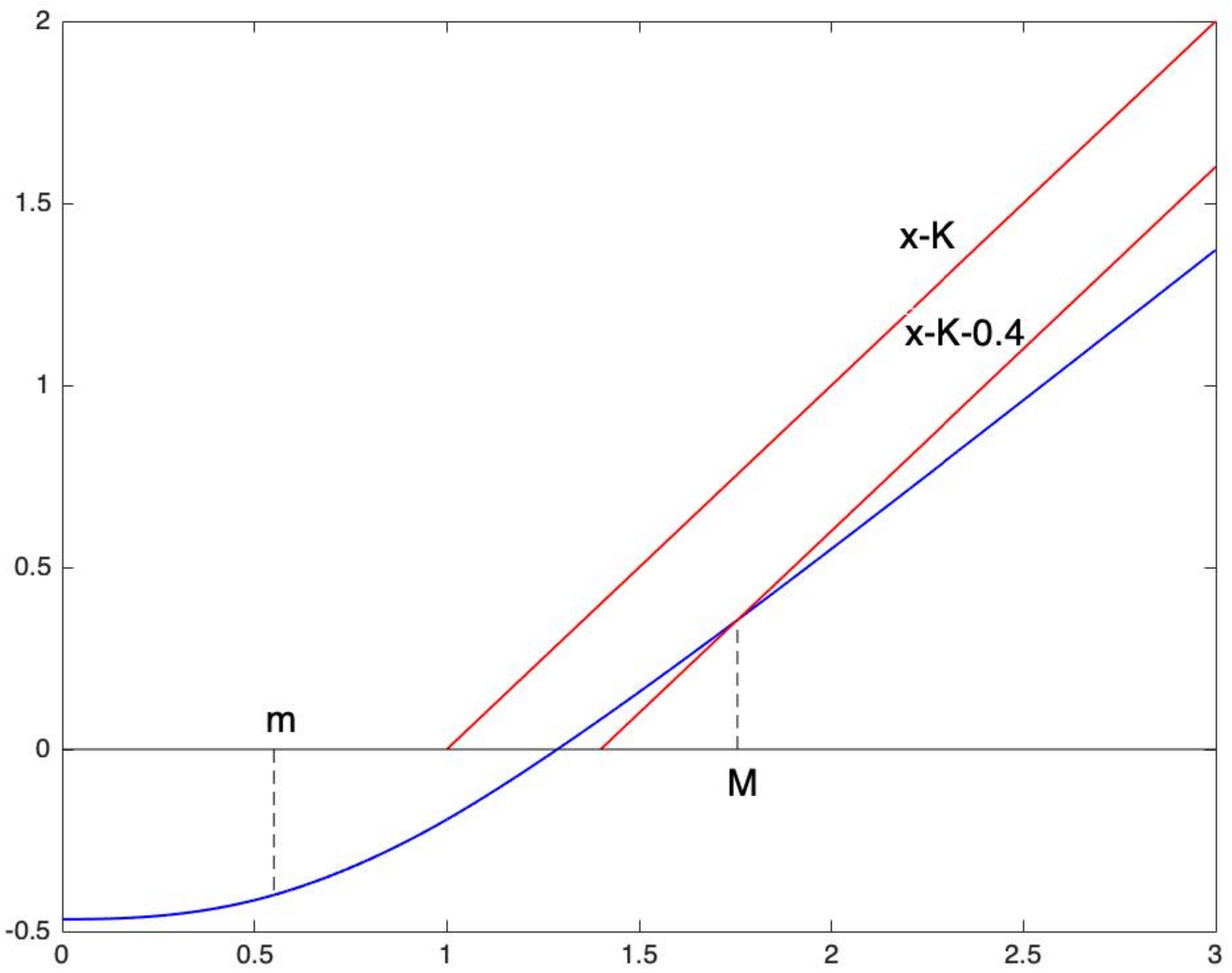}
        \caption{The value function $H$ in the case $(C(\ul)=5,C(\ol)=7)$. $b=0.4$.}
    \end{subfigure}%
    ~
    \begin{subfigure}[t]{0.45\textwidth}
        \centering
        \includegraphics[width=1\textwidth]{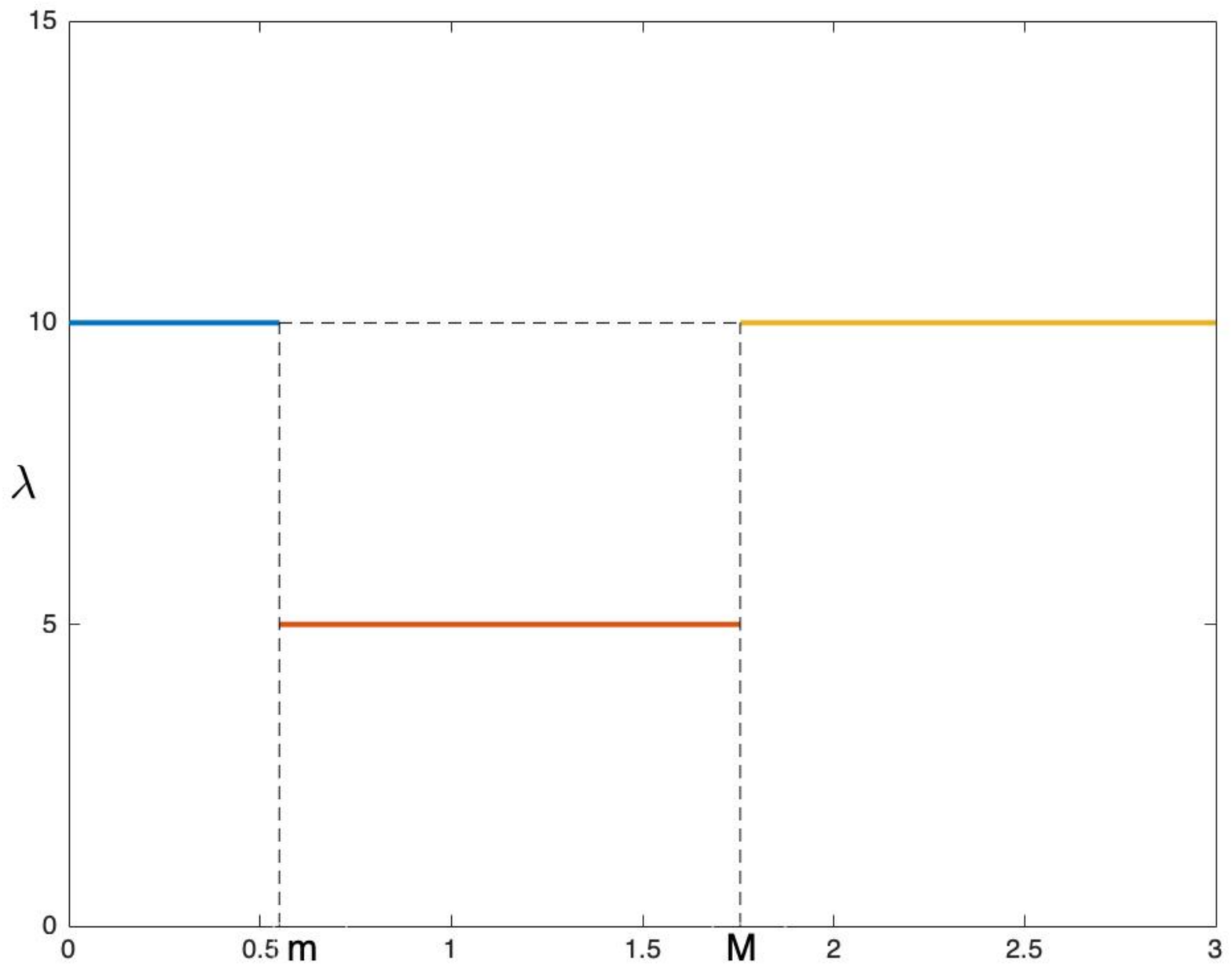}
        \caption{The optimal rate $\Lambda^*$ in the case $(C(\ul)=5,C(\ol)=7)$.}

    \end{subfigure}

    \begin{subfigure}[t]{0.45\textwidth}
        \includegraphics[width=1\textwidth]{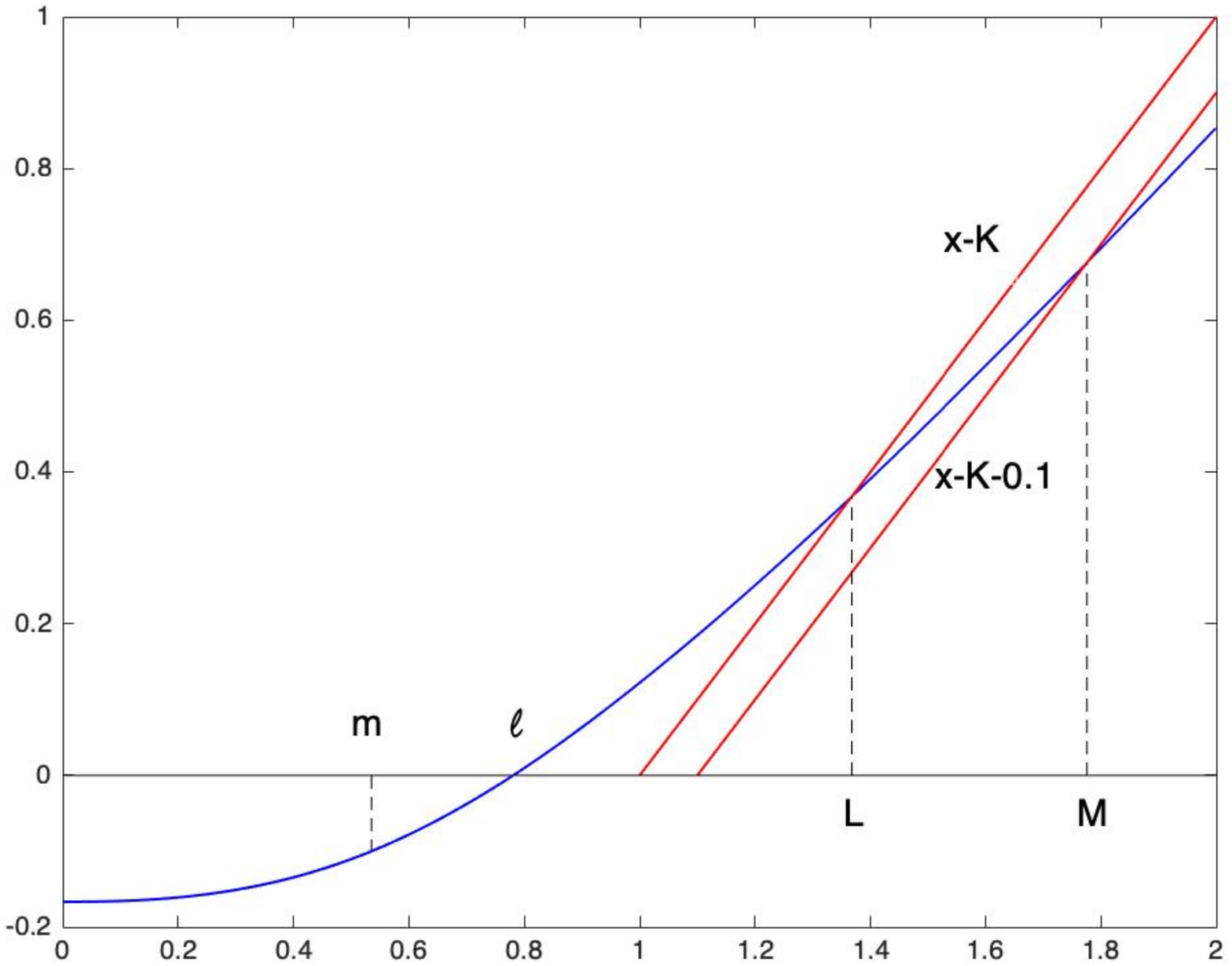}
        \caption{The value function $H$ in the case $(C(\ul)=2,C(\ol)=2.5)$. $b=0.1$.}
    \end{subfigure}%
    ~
    \begin{subfigure}[t]{0.45\textwidth}
        \centering
        \includegraphics[width=1\textwidth]{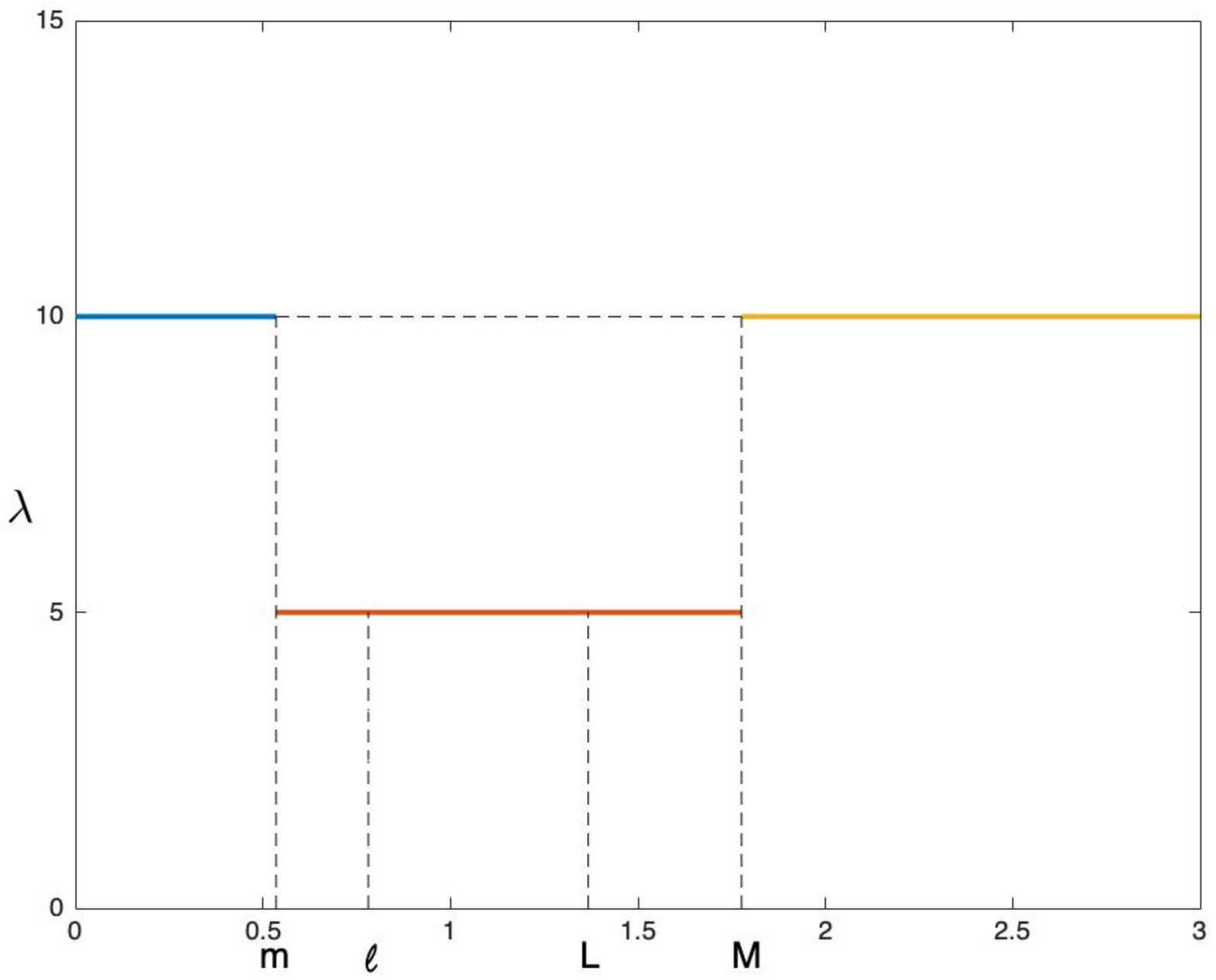}
        \caption{The optimal rate $\Lambda^*$ in the case $(C(\ul)=2,C(\ol)=2.5)$.}

    \end{subfigure}

    \caption{$(\beta,\mu,\sigma,K,\ul, \ol) =(5, 3 ,2,1,5,10)$.
    The left column plots the value function and the right column plots the optimal rate function. In each row $\frac{C(\ul)}{\ul + \beta} > \frac{C(\ul)}{\ul + \beta}$.  Near $x=0$ it is always preferable to choose the maximum possible rate process. Costs decrease as we move down the rows.
    }

    \label{fig:subinterval3}
 \end{figure*}

As a limiting special case suppose $\ol = \ul = \hl$ and that $C(\hl)=c \in [0,\infty)$. Then there is a single threshold $L$ to be determined and $H$ is of the form
\[ H(x) = \left\{ \begin{array}{ll} A x^{\theta} - \frac{c}{\beta} & x \leq L \\
                                   B x^{\hat{\phi}} + \frac{\hl}{\hl + \beta - \mu} x - \frac{(c+\hl K)}{\beta + \hl} & x > L
                                   \end{array} \right. \]
where $\hat{\phi}$ is the negative root of $Q_{\hl}(\cdot) = 0$.
The value matching condition $H(L)=(L-K)$ gives that $A = L^{-\theta}(L - K + \frac{c}{\beta})$ and
\[ B = L^{-\hat{\phi}}\left\{ \left( \frac{\beta - \mu}{\hl + \beta - \mu} \right) L + \frac{c - \beta K}{\beta + \hl} \right\}. \]
Then first order smooth fit at $L$ implies that
\[ L  = (\beta K - c) \left[ \frac{\theta}{\beta} - \frac{\hat{\phi}}{\beta + \hl} \right]\left\{ \theta - \frac{\hat{\phi}(\beta - \mu)}{\hl + \beta - \mu} - \frac{\hl}{\hl + \beta - \mu} \right\}^{-1} . \]
Note that if we take $c=0$ we recover exactly the expressions in (3.12) and (3.13) of Dupuis and Wang~\cite{DupuisWang:03}.

\section{Conclusion and discussion}
Our goal in this article is to extend the analysis of Dupuis and Wang~\cite{DupuisWang:03} who considered optimal stopping problems where the stopping time was constrained to lie in the event times of a Poisson process, to allow the agent to affect the frequency of those event times. The motivation was to model a form of illiquidity in trading and to consider problems in which the agent can exert effort in order to increase the opportunity set of candidate moments when the problem can terminate. This notion of effort is different to the idea in the financial economics literature of managers expending effort in order to change the dynamics of the underlying process, as exemplified by Sannikov~\cite{Sannikov:08}, but seems appropriate for the context.

Our work focuses on optimal stopping of an exponential Brownian motion under a perpetual call-style payoff, although it is clear given the work of Lempa~\cite{Lempa:07} how the analysis could be extended to other diffusion processes and other payoff functions. Nonetheless, even in this specific case we show how it is possible to generate a rich range of possible behaviours, depending on the choice of cost function. In our time-homogeneous, Markovian set-up, the rate of the Poisson process can be considered as a proxy for effort, and the problem can be cast in terms of this control variable. Then, the form of the solution depends crucially on the shape of the cost function, as a function of the rate of the inhomogeneous Poisson process.

One important quantity is the limiting value for large $\lambda$ of the average cost $\frac{C(\lambda)}{\lambda}$. If this limit is infinite, then the agent does not want to select very large rates for the Poisson process as they are too expensive. In this case we can replace $C$ with its convex minorant and solve the problem for that cost function. However, if $C$ is concave and the set of possible values for the rate process is unbounded then when the asset is sufficiently in the money, the agent wants to choose an infinite rate function, and thus to generate a stopping opportunity immediately. Choosing a very large rate function, albeit for a short time, incurs a cost equivalent to a fixed fee for stopping, and this is reflected in the form of the value function.

Another important quantity is the value of $C$ at zero. If a choice of zero stopping rate is feasible and incurs zero cost per unit time, then the agent always has a feasible, costless choice for the rate function, and the value function is non-negative. Then, when the asset price is close to zero we expect the agent to put no effort into searching for buyers, and to wait. However, if the cost of choosing a zero rate for the Poisson process is strictly positive, then the agent has an incentive to search for offers even when the asset price is small and the payoff is zero. When the agent receives an offer they accept, because this ends their obligation to pay costs. In this way we can have a range of optimal behaviours when the asset price is small.

When the range of possible rate processes includes zero and $C$ is strictly increasing, then the agent only exerts effort to generate selling opportunities in circumstances where they would accept those opportunities. The result is that the agent stops at the first event of the Poisson process, and the optimal stopping element of the problem is trivial.
However, an interesting feature arises when there is a lower bound on the admissible rate process. Then, the agent may receive unwanted offers, which they choose to decline. In this case the agent chooses whether to accept the first offer or to continue.

We model the cost function $C$ as increasing, which seems a natural requirement of the problem. (However, if $C$ is not increasing, we can introduce a largest increasing cost function which lies below $C$, and the value function for that problem will match the solution of the original problem.) We also assume that the interval of possible values for the rate process is closed (at any finite endpoints) and that $C$ is lower semi-continuous. Neither of these assumptions is essential although they do simplify the analysis. In particular, these assumptions ensure that the minimal cost is attained, and that we do not need to consider a sequence of approximating strategies and problems.

%Idea; include liquidity. Constraint on arrival of buyers.

%Range of possible behaviours.

%{\bf Increasing}, convex minorant of $C$. We assumed increasing from context, but clear? how to deal non-increasing.

%Two key regimes $\liminf_{\lambda \uparrow \infty} \frac{C(\lambda)}{\lambda}$ finite or not.

%What determines Behaviour near zero?

%Usually stop at opportunity, but in some cases not. Which cases.

%Extension to other processes.

\end{document}